\documentclass[12pt, reqno]{amsart}
\usepackage{amsmath, amsthm, amscd, amsfonts, amssymb, graphicx, color, mathrsfs}
\usepackage[bookmarksnumbered, colorlinks, plainpages]{hyperref}
\usepackage[all]{xy} 
\usepackage{slashed}

\newtheorem{theorem}{Theorem}[section]
\newtheorem{lemma}[theorem]{Lemma}

\newtheorem{corollary}[theorem]{Corollary}

\theoremstyle{definition}

\theoremstyle{remark}
\newtheorem{remark}[theorem]{Remark}
\numberwithin{equation}{section}

\begin{document}
\setcounter{page}{1}

\title[ Bj\"ork-Sj\"olin condition on graded Lie groups ]{Bj\"ork-Sj\"olin condition for strongly singular convolution operators on graded Lie groups}

\author[D. Cardona]{Duv\'an Cardona}
\address{
  Duv\'an Cardona:
  \endgraf
  Department of Mathematics: Analysis, Logic and Discrete Mathematics
  \endgraf
  Ghent University, Belgium
  \endgraf
  {\it E-mail address} {\rm duvanc306@gmail.com, duvan.cardonasanchez@ugent.be}
  }
  
\author[M. Ruzhansky]{Michael Ruzhansky}
\address{
  Michael Ruzhansky:
  \endgraf
  Department of Mathematics: Analysis, Logic and Discrete Mathematics
  \endgraf
  Ghent University, Belgium
  \endgraf
 and
  \endgraf
  School of Mathematical Sciences
  \endgraf
  Queen Mary University of London
  \endgraf
  United Kingdom
  \endgraf
  {\it E-mail address} {\rm michael.ruzhansky@ugent.be, m.ruzhansky@qmul.ac.uk}
  }

\thanks{The authors are supported  by the FWO  Odysseus  1  grant  G.0H94.18N:  Analysis  and  Partial Differential Equations and by the Methusalem programme of the Ghent University Special Research Fund (BOF)
(Grant number 01M01021). Michael Ruzhansky is also supported  by EPSRC grant 
EP/R003025/2.
}


\begin{abstract} In this work we extend the $L^1$-Bj\"ork-Sj\"olin theory of strongly  singular  convolution operators to arbitrary graded Lie groups.  Our criteria are presented in terms of the oscillating H\"ormander condition due to Bj\"ork and Sj\"olin   of the kernel of the operator, and the decay of its group Fourier transform is measured in terms of the infinitesimal representation of an arbitrary Rockland operator. The historical result by Bj\"ork and Sj\"olin is re-obtained in the case of the Euclidean space.
\end{abstract} 

\maketitle

\tableofcontents
\allowdisplaybreaks

\section{Introduction}

The aim of this manuscript is to  extend the theory of strongly singular integrals by Bj\"ork and Sj\"olin \cite{Bjork,Sjolin} to arbitrary graded Lie groups. This family of Lie groups includes Heisenberg type groups,  stratified groups, and are characterised between the family of nilpotent Lie groups by the existence of (Rockland operators) hypoelliptic left-invariant homogeneous partial differential operators in view of the Helffer and Nourrigat solution of the Rockland conjecture \cite{HelfferNourrigat}. 

Oscillating singular integrals arise as generalisations of the oscillating Fourier multipliers. In the euclidean setting they are used in PDE to estimate in the family of Sobolev spaces the hyperbolic  differential problems associated to the powers of elliptic operators, in particular of the fractional (positive) Laplacian $\Delta_x^{\frac{\gamma}{2}},$ where $0<\gamma<1.$   

In the Euclidean setting, oscillating Fourier multipliers are associated to symbols of the form
\begin{equation}\label{OSC:F:M}
    \widehat{K}(\xi)=\psi(\xi)\frac{e^{i|\xi|^a}}{|\xi|^{\frac{n\alpha}{2}}},\,\psi\in C^{\infty}(\mathbb{R}^n),\quad 0<a<1,
\end{equation} where $\psi$ vanishes near the origin and is equal to one for  $|\xi|$ large. It was proved by Wainger \cite{Wainger1965} that $K(x)$ is essentially equal to $c_n|x|^{-n-\lambda}e^{ic_{n}'|x|^{a'}},$ where
$  \lambda=\frac{n(a-\alpha)}{2(1-a)},$ and $ a'=\frac{a}{a-1}.$ From this one can deduce that $$|\nabla K(x)|\lesssim |x|^{-n-\lambda-1+a'}.$$
This gradient estimate shows that a kernel that satisfies \eqref{OSC:F:M} is outside of the theory of singular integrals due to Calder\'on and Zygmund  \cite{CalderonZygmund1952}. Nevertheless, the boundedness of singular integrals defined by kernels as in \eqref{OSC:F:M} was extensively investigated in the classical works of   Hardy \cite{Hardy1913}, 
Hirschman \cite{Hirschman1956} and Wainger  \cite{Wainger1965} until the end-points estimates proved by \cite{Fefferman1970,FeffermanStein1972}.  Further works on the subject in the setting of smooth manifolds  and beyond can be found in Seeger \cite{Seeger,Seeger1990,Seeger1991}, Seeger and Sogge \cite{SeegerSogge} and for the setting of Fourier integral operators, we refer the reader to Seeger, Sogge and Stein \cite{SSS} and Tao \cite{Tao}. 

In  \cite{Fefferman1970,FeffermanStein1972} Fefferman and Stein introduced a theory for oscillating Fourier multipliers which are convolution operators with singular kernels satisfying the condition 
\begin{equation}\label{FeffCond}
   \textnormal{A}(\theta):\,  \sup_{0<R<1}\Vert\, \smallint\limits_{|x|\geq 2R^{1-\theta}}|K(x-y)-K(x)|dx \Vert_{L^\infty(B(0,R),\,dy)}  <\infty,
\end{equation} for some $0\leq \theta<1,$ and  its Fourier transform has order $-n\theta/2,$ that is
\begin{equation}\label{decay}
  \textnormal{B}(\theta):\,  |\widehat{K}(\xi)|=O((1+|\xi|)^{-\frac{n\theta}{2}}),\quad 0\leq \theta<1.
\end{equation}
With $\theta=0,$ Fefferman-Stein's conditions agree with the one introduced  by H\"ormander \cite{Hormander1960} for the standard Calder\'on-Zygmund operators \cite{CalderonZygmund1952}. However, with $0<\theta<1,$ the conditions above by Fefferman and Stein also consider the oscillating kernels as in \eqref{OSC:F:M}.

The boundedness theory due to Fefferman and Stein can be summarised (by several reasons, including the real and complex interpolation theory of continuous linear operators on Lebesgue spaces) in the following theorem.
\begin{theorem}[Fefferman and Stein \cite{Fefferman1970,FeffermanStein1972}, 1970-1972]\label{Fefferman:Stein}
Assume that $K\in L^1_{loc}(\mathbb{R}^n\setminus\{0\})$ is a distribution with compact support satisfying the hypothesis $\textnormal{A}(\theta)$ and $\textnormal{B}(\theta)$ with $0\leq \theta<1.$  Then the convolution operator 
$$  T:f\mapsto f\ast K, $$ admits an extension of weak (1,1) type. Moreover, $T$ admits a bounded extension from the Hardy space $H^1(\mathbb{R}^n)$ into $L^1(\mathbb{R}^n).$
\end{theorem}

On the other hand, answering a question by Bj\"ork in \cite{Bjork}, Sj\"olin  \cite{Sjolin} developed the $L^1$-theory for the convolution operators $T:f\mapsto f\ast K$ where the kernel $K$ satisfies the two conditions $\textnormal{A}(\alpha)$ and $\textnormal{B}(\theta)$ given by
\begin{equation}\label{Sjolin:1}
   \textnormal{A}(\theta):\,  \sup_{0<R<b}\Vert\, \smallint\limits_{|x|\geq 2R^{1-\theta}}|K(x-y)-K(x)|dx \Vert_{L^\infty(B(0,R),\,dy)}  <\infty,
\end{equation} where $0<b< 1,$ and 
\begin{equation}\label{Sjolin:2}
  \textnormal{B}(\alpha):\,  |\widehat{K}(\xi)|=O((1+|\xi|)^{-\frac{n\alpha}{2}}),\quad 0\leq \alpha<1,
\end{equation}
where $0<\alpha<\theta<1.$  In the standard terminology of harmonic analysis, a convolution operator with  kernel satisfying the conditions $ \textnormal{A}(\theta)$ and $\textnormal{B}(\alpha)$ with  $0<\alpha<\theta<1$ is  called a strongly singular integral. The result in Sj\"olin  \cite{Sjolin} states the boundedness of this family of operators in $L^1(\mathbb{R}^n)$ as follows. Here, $\Delta_x=-\sum_{j=1}^n\partial_{x_j}^2$ is the positive Laplacian on $\mathbb{R}^n,$ and for any $s\in \mathbb{R},$ $L^1_{s}(\mathbb{R}^n)$ is the Sobolev space obtained from the closure of $C^{\infty}_0(\mathbb{R}^n)$ by the norm $\Vert f\Vert_{L^1_s}:=\Vert(1+\Delta_x)^{\frac{s}{2}}f\Vert_{L^1}.$ 

\begin{theorem}[Sj\"olin \cite{Sjolin}, 1976]
Assume that $K\in L^1_{loc}(\mathbb{R}^n\setminus\{0\})$ is a distribution with compact support satisfying the hypothesis $\textnormal{A}(\theta)$ and $\textnormal{B}(\alpha)$ with $0<\alpha< \theta<1.$ Then, $T:H^1(\mathbb{R}^n)\rightarrow L^1_{-\varkappa}(\mathbb{R}^n)$ extends to a bounded operator provided that
\begin{equation}
    \varkappa\geq {n(\theta-\alpha)}/{[n(1-\theta)+2]},
\end{equation} or equivalently, $$(1+\Delta_x)^{-\frac{\varkappa}{2}}T:H^{1}(\mathbb{R}^n)\rightarrow L^1(\mathbb{R}^n)$$ admits a bounded extension.
\end{theorem}
In the recent works \cite{CR20221,CR20223} the authors have generalised on graded Lie groups (with the Fourier transform criteria in terms of Rockland operators) the theory established by Fefferman and Stein in \cite{Fefferman1970,FeffermanStein1972}. 
The following extension of Theorem  \ref{Fefferman:Stein} has been obtained as part of the investigation done in \cite{CR20221,CR20223}.

\begin{theorem}[\cite{CR20221,CR20223}]\label{CR:0CZ:Graded:2022}
 Consider $G$ to be a  graded Lie group,  let $|\cdot|$ be a homogeneous quasi-norm on $G$ and let $Q$ be  its homogeneous dimension. Let $\mathcal{R}$ be a Rockland operator of homogeneous degree $\nu>0.$ Assume that the kernel $K$ of the convolution operator $T:f\mapsto f\ast K,$  satisfies the estimate
     \begin{equation}\label{Fourier:growth:1New}
        \sup_{\pi\in \widehat{G}}\Vert \widehat{K}(\pi) (1+\pi(\mathcal{R}))^{\frac{Q\theta}{2\nu}} \Vert_{\textnormal{op}}<\infty,
    \end{equation} 
    and the kernel condition   
    \begin{equation}\label{GS:CZ:cond:22}
        [K]_{H_{\infty,\theta}}':=\sup_{0<R<1}\sup_{|y|<R}  \smallint\limits_{|x|\geq 2R^{1-\theta}}|K(y^{-1}x)-K(x)|dx  <\infty.
    \end{equation} Then $T:H^1(G)\rightarrow L^1(G)$ extends to a bounded operator from the Hardy space $H^1(G)$ into $L^1(G)$. Moreover, $T:L^1(G)\rightarrow L^{1,\infty}(G)$ admits an extension of weak (1,1) type.
\end{theorem}
In this work we are going to  extend in our main Theorem \ref{Sjolin:Th:graded:groups} the conditions  $\textnormal{A}(\theta)$ and $\textnormal{B}(\alpha)$ of \eqref{Sjolin:1} and \eqref{Sjolin:2} due to   Bj\"ork and Sj\"olin to arbitrary graded Lie groups. To present the statement of the theorem we introduce some notations.

Here, for any graded Lie group $G,$ and $s\in \mathbb{R},$ $L^1_{s}(G)$ is the closure of $C^{\infty}_0(G)$ by the norm $$\Vert f\Vert_{L^1_s(G)}:=\Vert(1+\mathcal{R})^{\frac{s}{\nu}}f\Vert_{L^1},$$ where $\mathcal{R}$ is a positive Rockland operator of homogeneous degree $v>0,$ and  $L^{1,\infty}_{s}(G)$ is the weak-$L^{1}_{s}(G)$ Sobolev space defined by the semi-norm
$$ \Vert f \Vert_{ L^{1,\infty}_{s}(G)}:=\sup_{\lambda>0}\lambda|\{x\in G:|(1+\mathcal{R})^{\frac{s}{\nu}}f(x)|>\lambda\} <\infty. $$
The main result of this work is the following.
\begin{theorem}\label{Sjolin:Th:graded:groups} Consider $G$ to be a  graded Lie group,  let $|\cdot|$ be a homogeneous quasi-norm on $G$ and let $Q$ be  its homogeneous dimension. Let $\mathcal{R}$ be a Rockland operator of homogeneous degree $\nu>0.$ Let $K\in L^1_{\textnormal{loc}}(G\setminus \{e\})$ be a distribution of compact support and let $T:f\mapsto f\ast K,$ be the corresponding  integral operator associated to $K.$ Assume that for  $0<\alpha< \theta<1,$ $K$ satisfies the Fourier transform estimate
     \begin{equation}\label{A:alpha}
        \sup_{\pi\in \widehat{G}}\Vert \widehat{K}(\pi) (1+\pi(\mathcal{R}))^{\frac{Q\alpha}{2\nu}} \Vert_{\textnormal{op}}<\infty,
    \end{equation} 
    and the kernel condition   
    \begin{equation}
        [K]_{H_{\infty,\theta,b}}':=\sup_{0<R<b}\sup_{|y|<R}  \smallint\limits_{|x|\geq 2R^{1-\theta}}|K(y^{-1}x)-K(x)|dx  <\infty,
    \end{equation} where $0<b< 1.$
Then $T:H^1(G)\rightarrow L^1_{-\varkappa}(G)$ extends to a bounded operator provided that
\begin{equation}
    \varkappa\geq {Q(\theta-\alpha)}/{[Q(1-\theta)+2]},
\end{equation} or equivalently, $$(1+\mathcal{R})^{-\frac{\varkappa}{\nu}}T:H^{1}(G)\rightarrow L^1(G)$$ admits a bounded extension. Moreover, $T:L^1(G)\rightarrow L^{1,\infty}_{-\varkappa}(G)$ extends to a bounded operator, or equivalently,  
$$(1+\mathcal{R})^{-\frac{\varkappa}{\nu}}T:L^{1}(G)\rightarrow L^{1,\infty}(G)$$ 
admits an extension of weak $(1,1)$ type.
\end{theorem}
All this work will be dedicated to prove this statement. In Section \ref{preliminaries} we record the aspects of the Fourier analysis on graded Lie groups and the analysis of Rockland operators used in this work and finally, in Section \ref{proof:section} we prove Theorem \ref{Sjolin:Th:graded:groups}.

\section{Fourier analysis on graded groups}\label{preliminaries}

The notation and terminology of this paper on the analysis of homogeneous Lie groups are mostly taken 
from Folland and Stein \cite{FollandStein1982}. For the analysis of Rockland operators we will follow  \cite[Chapter 4]{FischerRuzhanskyBook}. 

\subsection{Homogeneous and graded Lie groups} 
    Let $G$ be a homogeneous Lie group. This means that $G$ is a connected and simply connected Lie group whose Lie algebra $\mathfrak{g}$ is endowed with a family of dilations $D_{r}^{\mathfrak{g}},$ $r>0,$ which are automorphisms on $\mathfrak{g}$  satisfying the following two conditions:
\begin{itemize}
\item For every $r>0,$ $D_{r}^{\mathfrak{g}}$ is a map of the form
$$ D_{r}^{\mathfrak{g}}=\textnormal{Exp}(\ln(r)A) $$
for some diagonalisable linear operator $A\equiv \textnormal{diag}[\nu_1,\cdots,\nu_n]$ on $\mathfrak{g}.$
\item $\forall X,Y\in \mathfrak{g}, $ and $r>0,$ $[D_{r}^{\mathfrak{g}}X, D_{r}^{\mathfrak{g}}Y]=D_{r}^{\mathfrak{g}}[X,Y].$ 
\end{itemize}
We call  the eigenvalues of $A,$ $\nu_1,\nu_2,\cdots,\nu_n,$ the dilations weights or weights of $G$.  The homogeneous dimension of a homogeneous Lie group $G$ is given by  $$ Q=\textnormal{\textbf{Tr}}(A)=\nu_1+\cdots+\nu_n.  $$
The dilations $D_{r}^{\mathfrak{g}}$ of the Lie algebra $\mathfrak{g}$ induce a family of  maps on $G$ defined via,
$$ D_{r}:=\exp_{G}\circ D_{r}^{\mathfrak{g}} \circ \exp_{G}^{-1},\,\, r>0, $$
where $\exp_{G}:\mathfrak{g}\rightarrow G$ is the usual exponential mapping associated to the Lie group $G.$ We refer to the family $D_{r},$ $r>0,$ as dilations on the group. If we write $rx=D_{r}(x),$ $x\in G,$ $r>0,$ then a relation on the homogeneous structure of $G$ and the Haar measure $dx$ on $G$ is given by $$ \smallint\limits_{G}(f\circ D_{r})(x)dx=r^{-Q}\smallint\limits_{G}f(x)dx. $$
    
A  Lie group is graded if its Lie algebra $\mathfrak{g}$ may be decomposed as the sum of subspaces $\mathfrak{g}=\mathfrak{g}_{1}\oplus\mathfrak{g}_{2}\oplus \cdots \oplus \mathfrak{g}_{s}$ such that $[\mathfrak{g}_{i},\mathfrak{g}_{j} ]\subset \mathfrak{g}_{i+j},$ and $ \mathfrak{g}_{i+j}=\{0\}$ if $i+j>s.$  Examples of such groups are the Heisenberg group $\mathbb{H}^n$ and more generally any stratified groups where the Lie algebra $ \mathfrak{g}$ is generated by $\mathfrak{g}_{1}$.  Here, $n$ is the topological dimension of $G,$ $n=n_{1}+\cdots +n_{s},$ where $n_{k}=\mbox{dim}\mathfrak{g}_{k}.$

A Lie algebra admitting a family of dilations is nilpotent, and hence so is its associated
connected, simply connected Lie group. The converse does not hold, i.e., not every
nilpotent Lie group is homogeneous  although they exhaust a large class, see \cite{FischerRuzhanskyBook} for details. Indeed, the main class of Lie groups under our consideration is that of graded Lie groups. A graded Lie group $G$ is a homogeneous Lie group equipped with a family of weights $\nu_j,$ all of them positive rational numbers. Let us observe that if $\nu_{i}=\frac{a_i}{b_i}$ with $a_i,b_i$ integer numbers,  and $b$ is the least common multiple of the $b_i's,$ the family of dilations 
$$ \mathbb{D}_{r}^{\mathfrak{g}}=\textnormal{Exp}(\ln(r^b)A):\mathfrak{g}\rightarrow\mathfrak{g}, $$
have integer weights,  $\nu_{i}=\frac{a_i b}{b_i}. $ So, in this paper we always assume that the weights $\nu_j,$ defining the family of dilations are non-negative integer numbers which allow us to assume that the homogeneous dimension $Q$ is a non-negative integer number. This is a natural context for the study of Rockland operators (see Remark 4.1.4 of \cite{FischerRuzhanskyBook}).

\subsection{Fourier analysis on nilpotent Lie groups}

Let $G$ be a simply connected nilpotent Lie group. Then the adjoint representation $\textnormal{ad}:\mathfrak{g}\rightarrow\textnormal{End}(\mathfrak{g})$ is nilpotent.  
Let us assume that $\pi$ is a continuous, unitary and irreducible  representation of $G,$ this means that,
\begin{itemize}
    \item $\pi\in \textnormal{Hom}(G, \textnormal{U}(H_{\pi})),$ for some separable Hilbert space $H_\pi,$ i.e. $\pi(xy)=\pi(x)\pi(y)$ and for the  adjoint of $\pi(x),$ $\pi(x)^*=\pi(x^{-1}),$ for every $x,y\in G.$
    \item The map $(x,v)\mapsto \pi(x)v, $ from $G\times H_\pi$ into $H_\pi$ is continuous.
    \item For every $x\in G,$ and $W_\pi\subset H_\pi,$ if $\pi(x)W_{\pi}\subset W_{\pi},$ then $W_\pi=H_\pi$ or $W_\pi=\emptyset.$
\end{itemize} Let $\textnormal{Rep}(G)$ be the set of unitary, continuous and irreducible representations of $G.$ The relation, {\small{
\begin{equation*}
    \pi_1\sim \pi_2\textnormal{ if and only if, there exists } A\in \mathscr{B}(H_{\pi_1},H_{\pi_2}),\textnormal{ such that }A\pi_{1}(x)A^{-1}=\pi_2(x), 
\end{equation*}}}for every $x\in G,$ is an equivalence relation and the unitary dual of $G,$ denoted by $\widehat{G}$ is defined via
$
    \widehat{G}:={\textnormal{Rep}(G)}/{\sim}.
$ Let us denote by $d\pi$ the Plancherel measure on $\widehat{G}.$ 
The Fourier transform of $f\in \mathscr{S}(G), $ (this means that $f\circ \textnormal{exp}_G\in \mathscr{S}(\mathfrak{g})$, with $\mathfrak{g}\simeq \mathbb{R}^{\dim(G)}$) at $\pi\in\widehat{G},$ is defined by 
\begin{equation*}
    \widehat{f}(\pi)=\smallint\limits_{G}f(x)\pi(x)^*dx:H_\pi\rightarrow H_\pi,\textnormal{   and   }\mathscr{F}_{G}:\mathscr{S}(G)\rightarrow \mathscr{S}(\widehat{G}):=\mathscr{F}_{G}(\mathscr{S}(G)).
\end{equation*}

If we identify one representation $\pi$ with its equivalence class, $[\pi]=\{\pi':\pi\sim \pi'\}$,  for every $\pi\in \widehat{G}, $ the Kirillov trace character $\Theta_\pi$ defined by  $$(\Theta_{\pi},f):
=\textnormal{\textbf{Tr}}(\widehat{f}(\pi)),$$ is a tempered distribution on $\mathscr{S}(G).$ In particular, the identity
$
    f(e_G)=\smallint\limits_{\widehat{G}}(\Theta_{\pi},f)d\pi,
$ 
implies the Fourier inversion formula $f=\mathscr{F}_G^{-1}(\widehat{f}),$ where
\begin{equation*}
    (\mathscr{F}_G^{-1}\sigma)(x):=\smallint\limits_{\widehat{G}}\textnormal{\textbf{Tr}}(\pi(x)\sigma(\pi))d\pi,\,\,x\in G,\,\,\,\,\mathscr{F}_G^{-1}:\mathscr{S}(\widehat{G})\rightarrow\mathscr{S}(G),
\end{equation*}is the inverse Fourier  transform. In this context, the Plancherel theorem takes the form $\Vert f\Vert_{L^2(G)}=\Vert \widehat{f}\Vert_{L^2(\widehat{G})}$,  where  $$L^2(\widehat{G}):=\smallint\limits_{\widehat{G}}H_\pi\otimes H_{\pi}^*d\pi,$$ is the Hilbert space endowed with the norm: $\Vert \sigma\Vert_{L^2(\widehat{G})}=(\int_{\widehat{G}}\Vert \sigma(\pi)\Vert_{\textnormal{HS}}^2d\pi)^{\frac{1}{2}}.$

\subsection{Homogeneous linear operators and Rockland operators} A linear operator $T:C^\infty(G)\rightarrow \mathscr{D}'(G)$ is homogeneous of  degree $\nu\in \mathbb{C}$ if for every $r>0$ the equality 
\begin{equation*}
T(f\circ D_{r})=r^{\nu}(Tf)\circ D_{r}
\end{equation*}
holds for every $f\in \mathscr{D}(G). $
If for every representation $\pi\in\widehat{G},$ $\pi:G\rightarrow U({H}_{\pi}),$ we denote by ${H}_{\pi}^{\infty}$ the set of smooth vectors, that is, the space of elements $v\in {H}_{\pi}$ such that the function $x\mapsto \pi(x)v,$ $x\in \widehat{G},$ is smooth,  a Rockland operator is a left-invariant differential operator $\mathcal{R}$ which is homogeneous of positive degree $\nu=\nu_{\mathcal{R}}$ and such that, for every unitary irreducible non-trivial representation $\pi\in \widehat{G},$ $\pi({R})$ is injective on ${H}_{\pi}^{\infty};$ $\sigma_{\mathcal{R}}(\pi)=\pi(\mathcal{R})$ is the symbol associated to $\mathcal{R}.$ It coincides with the infinitesimal representation of $\mathcal{R}$ as an element of the universal enveloping algebra. It can be shown that a Lie group $G$ is graded if and only if there exists a differential Rockland operator on $G.$ If the Rockland operator is formally self-adjoint, then $\mathcal{R}$ and $\pi(\mathcal{R})$ admit self-adjoint extensions on $L^{2}(G)$ and ${H}_{\pi},$ respectively. Now if we preserve the same notation for their self-adjoint
extensions and we denote by $E$ and $E_{\pi}$  their spectral measures, we will denote by
$$ f(\mathcal{R})=\smallint\limits_{-\infty}^{\infty}f(\lambda) dE(\lambda),\,\,\,\textnormal{and}\,\,\,\pi(f(\mathcal{R}))\equiv f(\pi(\mathcal{R}))=\smallint\limits_{-\infty}^{\infty}f(\lambda) dE_{\pi}(\lambda), $$ the functions defined by the functional calculus. 
In general, we will reserve the notation $\{dE_A(\lambda)\}_{0\leq\lambda<\infty}$ for the spectral measure associated with a positive and self-adjoint operator $A$ on a Hilbert space $H.$ 

We now recall a lemma on dilations on the unitary dual $\widehat{G},$ which will be useful in our analysis of spectral multipliers.   For the proof, see Lemma 4.3 of \cite{FischerRuzhanskyBook}.
\begin{lemma}\label{dilationsrepre}
For every $\pi\in \widehat{G}$ let us define  
\begin{equation}\label{dilations:repre}
\pi^{(r)}(x)=  D_{r}(\pi)(x)= (r\cdot \pi)(x)=\pi(r\cdot x)\equiv \pi(D_r(x)),  
\end{equation}
 for every $r>0$ and all $x\in G.$ Then, if $f\in L^{\infty}(\mathbb{R})$ then $f(\pi^{(r)}(\mathcal{R}))=f({r^{\nu}\pi(\mathcal{R})}).$
\end{lemma}
\begin{remark}{{For instance, for any $\alpha\in \mathbb{N}_0^n,$ and for an arbitrary family $X_1,\cdots, X_n,$ of left-invariant  vector-fields we will use the notation
\begin{equation}
    [\alpha]:=\sum_{j=1}^n\nu_j\alpha_j,
\end{equation}for the homogeneity degree of the operator $X^{\alpha}:=X_1^{\alpha_1}\cdots X_{n}^{\alpha_n},$ whose order is $|\alpha|:=\sum_{j=1}^n\alpha_j.$}}
\begin{remark}\label{DilationLebesgue}
By considering the dilation $r\cdot x=D_{r}(x),$ $x\in G,$ $r>0,$ then a relation  between the homogeneous structure of $G$ and the Haar measure $dx$ on $G$ is given by (see \cite[Page 100]{FischerRuzhanskyBook}) $$ \smallint\limits_{G}(f\circ D_{r})(x)dx=r^{-Q}\smallint\limits_{G}f(x)dx. $$
Note that if $f_{r}:=r^{-Q}f(r^{-1}\cdot),$ then 
\begin{equation}\label{Eq:dilatedFourier}
    \widehat{f}_{r}(\pi)=\smallint\limits_{G}r^{-Q}f(r^{-1}\cdot x)\pi(x)^*dx=\smallint\limits_{G}f(y)\pi(r\cdot y)^{*}dy=\widehat{f}(r\cdot \pi),
\end{equation}for any $\pi\in \widehat{G}$ and all $r>0,$ with $(r\cdot \pi)(y)=\pi(r\cdot y),$ $y\in G,$ as in \eqref{dilations:repre}.
\end{remark}
\end{remark} 

\section{Proof of the main theorem}\label{proof:section}

We will start our analysis for the proof of Theorem \ref{Sjolin:Th:graded:groups} by analysing the operator
$ 
\mathfrak{G}_{a}:=\left(\frac{\mathcal{R}}{1+\mathcal{R}}\right)^{\frac{a}{\nu}}   , 
$ 
for any $a>0,$ as a convolution operator with a finite measure. 
This analysis will be addressed in Lemma \ref{Lemma:quotient:Rockland} below where we extend an observation done e.g. in  Stein \cite[Page 133]{Stein1970} in  the case of the Laplace operator to general Rockland operators. During this work we will denote by $\mathcal{B}_{s}$ to the right convolution kernel of the operator $(1+\mathcal{R})^{-\frac{s}{\nu}},$ for any $s\in \mathbb{R}.$

\subsection{The quotient between the Riesz and the Bessel potential}
There is an intimate connection between the Bessel potential and the Riesz potentials of Rockland operators. This affinity between the two is given in precision in the following lemma.

\begin{lemma}\label{Lemma:quotient:Rockland}
Let $\alpha>0,$ and let $\mathcal{R}$ be a Rockland operator on $G$ of homogeneity degree $\nu>0.$ There exists a finite measure $\mu_\alpha$ on $G$ such that its Fourier transform is given by
\begin{equation}\label{mu:alpha}
    \widehat{\mu}_\alpha(\pi)=\left(\frac{\pi(\mathcal{R})}{1+\pi(\mathcal{R})}\right)^{\frac{\alpha}{\nu}},\,\pi\in\widehat{G}.
\end{equation}
\end{lemma}
\begin{proof}For the proof of Lemma \ref{Lemma:quotient:Rockland} let us use the expansion
\begin{equation}
    (1-t)^{\alpha/\nu}=1+\sum_{m=1}^{\infty}A_{m,\alpha/\nu}t^m,\,\,|t|<1,\,\sum_{m=1}^\infty|A_{m,\alpha/\nu}|<\infty,
\end{equation}which is still valid when $t\rightarrow 1^{-}$ because $(1-t)^{\alpha/\nu}$ remains bounded  for $\alpha>0$. Let $dE_{\pi(\mathcal{R})}$ be the spectral measure of the operator $\pi(\mathcal{R}).$ With $t=\frac{1}{1+\lambda},$ $\lambda\geq 0,$ we have that
\begin{align*}
    \left(\frac{\lambda}{1+\lambda}\right)^{\frac{\alpha}{\nu}}= \left(1-\frac{1}{1+\lambda}\right)^{\frac{\alpha}{\nu}}=1+\sum_{m=1}^{\infty}A_{m,\alpha/\nu}(1+\lambda)^{-m},
\end{align*}and then the functional calculus of the operator $\pi(\mathcal{R})$ implies that
\begin{align*}
    \left(\frac{\pi(\mathcal{R})}{1+\pi(\mathcal{R})}\right)^{\frac{\alpha}{\nu}} &=\smallint_{0}^{\infty} \left(\frac{\lambda}{1+\lambda}\right)^{\frac{\alpha}{\nu}}dE_{\pi(\mathcal{R})}(\lambda)=1+\sum_{m=1}^{\infty}\smallint_{0}^{\infty}A_{m,\alpha/\nu}(1+\lambda)^{-m}dE_{\pi(\mathcal{R})}(\lambda)\\
   &=1+\sum_{m=1}^{\infty}A_{m,\alpha/\nu}(1+\pi(\mathcal{R}))^{-m}.
\end{align*}In consequence, the required measure $\mu_\alpha$ is given by
\begin{equation}
    \mu_\alpha=\delta+\sum_{m=1}^{\infty}A_{m,\alpha/\nu}\mathcal{B}_{m\nu}(x)dx.
\end{equation}Indeed, note that $\mu_\alpha$ satisfies that $\widehat{\mu}_\alpha(\pi)= \left({\pi(\mathcal{R})}/[{1+\pi(\mathcal{R})]}\right)^{\frac{\alpha}{\nu}},$ $\pi\in\widehat{G}.$  The proof of Lemma \ref{Lemma:quotient:Rockland} is complete.
\end{proof}
\begin{corollary} Let $\alpha>0,$ and let $\mathcal{R}$ be a Rockland operator on $G$ of homogeneity degree $\nu>0.$ Then the operator
\begin{equation}
    \left(\frac{\mathcal{R}}{1+\mathcal{R}}\right)^{\frac{\alpha}{\nu}}:L^p(G)\rightarrow L^p(G),
\end{equation}extends to a bounded operator for all $1\leq p\leq \infty.$

\end{corollary}
\begin{proof}
 The action of the operator $\left(\frac{\mathcal{R}}{1+\mathcal{R}}\right)^{\frac{\alpha}{\nu}}$ on functions in $L^1(G)$ is obtained from the right convolution with the measure $\mu_\alpha$ in \eqref{mu:alpha}. So, $\left(\frac{\mathcal{R}}{1+\mathcal{R}}\right)^{\frac{\alpha}{\nu}}$ is bounded from $L^1(G)$ into $L^1(G).$ By the duality argument $\left(\frac{\mathcal{R}}{1+\mathcal{R}}\right)^{\frac{\alpha}{\nu}}$ is bounded from $L^\infty(G)$ into $L^\infty(G).$ Because $\left(\frac{\mathcal{R}}{1+\mathcal{R}}\right)^{\frac{\alpha}{\nu}}$ is bounded on $L^2(G),$ the Marcinkiewicz interpolation theorem implies  the boundedness of $ \left(\frac{\mathcal{R}}{1+\mathcal{R}}\right)^{\frac{\alpha}{\nu}}:L^p(G)\rightarrow L^p(G),$ for all $1\leq p\leq \infty.$
\end{proof}

\subsection{Boundedness of strongly singular integral operators} We are going to prove our main Theorem \ref{Sjolin:Th:graded:groups}. For this, we precise the notations.
\begin{itemize}
    \item[-] Consider $G$ to be a  graded Lie group,  let $|\cdot|$ be a homogeneous quasi-norm on $G$ and let $Q$ be  its homogeneous dimension.
    \item[-] Let $\mathcal{R}$ be a Rockland operator of homogeneous degree $\nu>0.$ Let $K\in L^1_{\textnormal{loc}}(G\setminus \{e\})$ and $T:f\mapsto f\ast K,$ the corresponding  integral operator associated to $K.$
   \item[-] Assume that for  $0<\alpha<\theta<1,$ $K$ satisfies the Fourier transform estimate
     \begin{equation}\label{A:alpha:2:2}
        \sup_{\pi\in \widehat{G}}\Vert \widehat{K}(\pi) (1+\pi(\mathcal{R}))^{\frac{Q\alpha}{2\nu}} \Vert_{\textnormal{op}}<\infty,
    \end{equation} 
    and the kernel condition   
    \begin{equation}
        [K]_{H_{\infty,\theta,b}}':=\sup_{|y|<b}  \smallint\limits_{|x|\geq 2R^{1-\theta}}|K(y^{-1}x)-K(x)|dx  <\infty
,\,0<b<1.    \end{equation}
\end{itemize}
We are going to prove that $T:H^1(G)\rightarrow L^1_{-\varkappa}(G)$ extends to a bounded operator provided that
\begin{equation}
    \varkappa\geq {Q(\theta-\alpha)}/{[Q(1-\theta)+2]},
\end{equation} or equivalently, that $$(1+\mathcal{R})^{-\frac{\varkappa}{\nu}}T:H^{1}(G)\rightarrow L^1(G)$$ admits a bounded extension. In the same way, we have to prove that 
$$(1+\mathcal{R})^{-\frac{\varkappa}{\nu}}T:L^{1}(G)\rightarrow L^{1,\infty}(G)$$ 
admits a bounded extension, which proves that $T:L^1(G)\rightarrow L^{1,\infty}_{-\varkappa}(G)$ extends to a bounded operator.
\begin{proof}[Proof of Theorem \ref{Sjolin:Th:graded:groups}]
It is suffice to consider the critical case
$$  \varkappa:=\frac{Q(\theta-\alpha)}{Q(1-\theta)+2}. $$
Indeed, having proved the boundedness of $T:H^1(G)\rightarrow L^1_{-\varkappa}(G)$ and of $T:L^1(G)\rightarrow L^{1,\infty}_{-\varkappa}(G),$   for any $\varkappa'>\varkappa=\frac{Q(\theta-\alpha)}{Q(1-\theta)+2},$ we have the continuous inclusions
$$ L^{1}_{-\varkappa}(G)\hookrightarrow L^{1}_{-\varkappa'}(G),\,L^{1,\infty}_{-\varkappa}(G)\hookrightarrow L^{1,\infty}_{-\varkappa'}(G)$$ implying also the existence of the bounded extensions  $T:H^1(G)\rightarrow L^1_{-\varkappa'}(G)$ and $T:L^1(G)\rightarrow L^{1,\infty}_{-\varkappa'}(G).$

Let us choose $\phi\in C^{\infty}_0(G)$ so that
\begin{equation}\label{the:auxiliar:phi}
 \phi\geq 0,\quad \textnormal{  with }   \textnormal{supp}[\phi]\subset\{x\in G: 1/2<|x|<2\},\,\,\phi(x)=1 \textnormal{ if } \frac{3}{4}<|x|<1,
\end{equation} and such that $$\forall x\in G, \,|x|<1,\,\,\sum_{k=0}^{\infty}\phi(2^k\cdot x)=1.$$ Also, for $x\in G,$ define $\phi_k(x):=\phi(2^k\cdot x)$ and $$\psi(x)=\sum_{k=0}^\infty\phi_{k}(x).$$ 
The kernel of the operator $(1+\mathcal{R})^{-\frac{\varkappa}{\nu}}T$ is given by $K\ast \mathcal{B}_{\varkappa}.$ Let us use the decomposition
$$ K\ast \mathcal{B}_{\varkappa}=K\ast(\mathcal{B}_{\varkappa}\psi)+K\ast(\mathcal{B}_{\varkappa}(1-\psi))=K_1+K_2,\,K_1:=K\ast(\mathcal{B}_{\varkappa}\psi).  $$

Let us prove that  $K_2=K\ast (\mathcal{B}_{\varkappa}(1-\psi)) \in \mathscr{S}(G)$ is a smooth function in $L^1(G).$ Indeed, from the properties of $\phi,$ we have that $\psi(x)=1$ for all $x\in G$ with $|x|<1,$ and then $1-\psi(x)\equiv 0$ when $|x|<1.$ 

On the other hand, the function $\mathcal{B}_{\varkappa}$ decreases rapidly for $|x|\geq 1.$ Indeed, for any $N\in \mathbb{N},$ there is $C_N>0,$ such that  (see \cite[Theorem 5.4.1]{FischerRuzhanskyBook})
\begin{align*}
    |\mathcal{B}_{\varkappa}(x)|\leq C_N|x|^{-N},\,|x|>1.
\end{align*}
Since, $x\in \textnormal{supp}(\phi_k)$ implies that $|2^k\cdot x|\in (1/2,2),$ then 
\begin{equation}
\forall k\geq 0,\,    \textnormal{supp}(\phi_k)\subset \{x\in G:2^{-k-1}<|x|<2^{-k+1}\} .
\end{equation}So, for any $x\in G:\,|x|>2,$ $\psi(x)=\sum_{k=0}^{\infty}\phi(2^k\cdot x)=0.$ 

In conclusion the function $(1-\psi)\mathcal{B}_\varkappa$ 
has its support in the complement of the set $\{x\in G:|x|<1\}$ and  $(1-\phi)\mathcal{B}_\varkappa \in L^1(G)\cap C^{\infty}(\{x\in G:|x|>1\}).$

So, the left convolution operator $T_{K_2}$ associated to $K_2$ is bounded from $L^1(G)$ into $L^1(G).$ The embedding $H^{1}(G)\hookrightarrow L^{1}(G)$ implies that $T_{K_2}:H^{1}(G)\hookrightarrow L^1(G)$ is bounded.  Note also that the boundedness of $T_{K_2}$ from $L^{1}(G)$ into $L^1(G)$ implies its boundedness from $L^{1}(G)$ into $L^{1,\infty}(G)$ in view of the inclusion $L^{1}(G)\hookrightarrow L^{1,\infty}(G).$

Now, to continue with the proof it  suffices to demonstrate the boundedness of $T_{K_1}=T-T_{K_2}$ from $H^{1}(G)$ into $ L^1(G),$ and from $L^{1}(G)$ into $L^{1,\infty}(G).$ For this, we will prove that $K_1=K\ast (\mathcal{B}_\varkappa\psi)$ satisfies the conditions
\begin{equation}\label{Fourier:growth:1New:1}
        \sup_{\pi\in \widehat{G}}\Vert \widehat{K}_1(\pi) (1+\pi(\mathcal{R}))^{\frac{Q a}{2\nu}} \Vert_{\textnormal{op}}<\infty,
    \end{equation} 
    and the kernel condition   
    \begin{equation}\label{GS:CZ:cond:22:2}
        [K]_{H_{\infty,\theta}}':=\sup_{0<R<1}\sup_{|y|<R}  \smallint\limits_{|x|\geq 2R^{1-a}}|K_1(y^{-1}x)-K_1(x)|dx  <\infty,
    \end{equation}
with 
\begin{equation}
    a:=\frac{Q\alpha(1-\theta)+2\theta}{Q(1-\theta)+2}\in (0,1).
\end{equation}
Note also that the hypothesis  $\alpha<\theta,$ implies that $Q\alpha(1-\theta)+2\theta<Q\theta(1-\theta)+2\theta=\theta[Q(1-\theta)+2]$ which (by dividing both sides of this inequality  by $Q(1-\theta)+2$) implies the estimate
\begin{align}\label{a:less:than:theta}
    0<a=\frac{Q\alpha(1-\theta)+2\theta}{Q(1-\theta)+2}<\theta,
\end{align} allowing the use of Theorem \ref{CR:0CZ:Graded:2022}.

For the proof of \eqref{Fourier:growth:1New:1} let use that  $\psi$ vanishes for $|x|>2.$ So, $\psi\mathcal{B}_\varkappa$ is of compact support of $G,$ and for all $r\in \mathbb{R},$
$$ (1+\mathcal{R})^{r/\nu}[\psi\mathcal{B}_\varkappa]\in L^1(G).  $$
Indeed, for any $r\in \mathbb{R},$ $(1+\mathcal{R})^{r/\nu}$ maps the Schwartz space $\mathscr{S}(G)$ into the Schwartz space $\mathscr{S}(G).$
So, for all $r\in \mathbb{R},$ $(1+\mathcal{R})^{r/\nu}[\psi\mathcal{B}_\varkappa]$ is the right-convolution kernel of a bounded operator on $L^2(G).$ Indeed, the Hausdorff-Young inequality gives
\begin{equation}
    \forall f\in L^2(G),\,\,\Vert f\ast (1+\mathcal{R})^{r/\nu}[\psi\mathcal{B}_\varkappa] \Vert_{L^2(G)}\leq \Vert f\Vert_{L^2(G)}\Vert (1+\mathcal{R})^{r/\nu}[\psi\mathcal{B}_\varkappa] \Vert_{L^1(G)}.
\end{equation}However, the Plancherel theorem indicates that for any $r\in \mathbb{R},$ 
 $$ \sup_{\pi\in \widehat{G}}\Vert \mathscr{F}[(1+\mathcal{R})^{r/\nu}[\psi\mathcal{B}_\varkappa]](\pi)\Vert_{\textnormal{op}}=\sup_{\pi\in \widehat{G}}\Vert (1+\pi(\mathcal{R}))^{r/\nu}\widehat{\psi\mathcal{B}}_\varkappa(\pi)\Vert_{\textnormal{op}}<\infty.  $$
In a similar way, we have that 
$$\forall r\in \mathbb{R},\, \sup_{\pi\in \widehat{G}}\Vert\widehat{\psi\mathcal{B}}_\varkappa(\pi) (1+\pi(\mathcal{R}))^{r/\nu}\Vert_{\textnormal{op}}<\infty.  $$
In consequence, for any $s>0,$
\begin{align*}
    &\sup_{\pi\in \widehat{G}}\Vert \widehat{K}_1(\pi) (1+\pi(\mathcal{R}))^{\frac{Q a}{2\nu}} \Vert_{\textnormal{op}} \\
    &= \sup_{\pi\in \widehat{G}}\Vert \widehat{\mathcal{B}_\varkappa\psi}(\pi)\widehat{K}(\pi) (1+\pi(\mathcal{R}))^{\frac{Q a}{2\nu}}\Vert_{\textnormal{op}}\\
    &= \sup_{\pi\in \widehat{G}}\Vert \widehat{\mathcal{B}_\varkappa\psi}(\pi)(1+\pi(\mathcal{R}))^{s/\nu}(1+\pi(\mathcal{R}))^{-s/\nu}\widehat{K}(\pi) (1+\pi(\mathcal{R}))^{\frac{Q a}{2\nu}}\Vert_{\textnormal{op}}\\
    &\leq \sup_{\pi\in \widehat{G}}\Vert \widehat{\mathcal{B}_\varkappa\psi}(\pi)(1+\pi(\mathcal{R}))^{s/\nu}\Vert_{\textnormal{op}} \Vert(1+\pi(\mathcal{R}))^{-s/\nu}\widehat{K}(\pi) (1+\pi(\mathcal{R}))^{\frac{Q a}{2\nu}}\Vert_{\textnormal{op}}\\
    &\lesssim_{s}  \Vert(1+\pi(\mathcal{R}))^{-s/\nu}\Vert_{\textnormal{op}}\Vert\widehat{K}(\pi) (1+\pi(\mathcal{R}))^{\frac{Q a}{2\nu}}\Vert_{\textnormal{op}}<\infty
\end{align*}which demonstrate  \eqref{Fourier:growth:1New:1}.
Now, we are going  to prove \eqref{GS:CZ:cond:22:2}. Define
$$ G_{\varkappa,k}:=\phi_k \mathcal{B}_\varkappa.  $$  Then, $$K\ast (\psi \mathcal{B}_\varkappa)=\sum_{k=0}^{\infty}K\ast G_{\varkappa,k}.$$
To do so, take $y\in G$ such that $|y|<\min\{b,\frac{1}{2}\}.$ For any $k\in \mathbb{N}_0,$ let
\begin{equation}
    I_k:=\smallint\limits_{|x|>2|y|^{1-a}}|K\ast G_{\varkappa,k}(y^{-1}x)-K\ast G_{\varkappa,k}(x)|dx.
\end{equation} Now, let us make an analysis of the last integral above when $t\in \textnormal{supp}(G_{\varkappa,k}).$  In that case $|2^{k}\cdot t|=2^{k}|t|\in (1/2,2)$ that is $2^{-k-1}<|t|<2^{-k+1}.$
Note that the changes of variables $z=xt^{-1}$ implies the inequalities
\begin{align*}
    I_k&=\smallint\limits_{|x|>2|y|^{1-a}}|K\ast G_{\varkappa,k}(y^{-1}x)-K\ast G_{\varkappa,k}(x)|dx\\
    &=\smallint\limits_{|x|>2|y|^{1-a}}|\smallint\limits_{G} (K(y^{-1}xt^{-1}) G_{\varkappa,k}(t)-K(xt^{-1}) G_{\varkappa,k}(t))dt|dx\\
    &\leq  \smallint\limits_{G} |G_{\varkappa,k}(t)| \smallint\limits_{|x|>2|y|^{1-a}} |K(y^{-1}xt^{-1}) -K(xt^{-1})|dxdt\\
    &\leq  \smallint\limits_{G} |G_{\varkappa,k}(t)| \smallint\limits_{|zt|>2|y|^{1-a}} |K(y^{-1}z)-K(z)|dzdt.
\end{align*}

So, we have proved the estimate
\begin{equation}\label{case1:deep:analysis}
  I_k\leq   \smallint\limits_{G}| G_{\varkappa,k}(t)| \smallint\limits_{|zt|>2|y|^{1-a}} |K(y^{-1}z)-K(z)|dz\,dt,
\end{equation}  where $|y|\leq b < 1.$ To continue, let us estimate the integral
$$\|G_{\varkappa,k}\|_{L^1(G)}=   \smallint\limits_{G} |G_{\varkappa,k}(t)|dt.  $$ First, observe that $\mathcal{B}_{\varkappa}$ is the right-convolution kernel of the pseudo-differential operator $(1+\mathcal{R})^{-\frac{\varkappa}{\nu}}\in \Psi^{-\varkappa}_{1,0}(G\times \widehat{G}).$ Note that $0<\varkappa<Q,$ which can be proved by observing that $$Q(\theta-\alpha)<2Q<Q^{2}(1-\theta)+2Q=Q(Q(1-\theta)+2)$$ implying that
 $\varkappa=Q(\theta-\alpha)/[Q(1-\theta)+2]<Q.$ So, $\mathcal{B}_{\varkappa}$ satisfies the estimate (see \cite[Theorem 5.4.1]{FischerRuzhanskyBook})
 $$  |\mathcal{B}_{\varkappa}(t)|\leq C_{\varkappa}|t|^{-(Q-\varkappa)},\,|t|\lesssim 1. $$
 In consequence the change of variable $u=2^{k}t$ has the effect in the Haar measure $du=2^{kQ}dt$ and then $dt=2^{-kQ}du$, implying the following estimates
 $$
    \smallint\limits_{G}| G_{\varkappa,k}(t)|dt=\smallint\limits_{G}| \mathcal{B}_{\varkappa}(t)\phi(2^{k}t)|dt\lesssim  \smallint\limits_{|2^{k}t|<2}| \mathcal{B}_{\varkappa}(t)\phi(2^{k}t)|dt\lesssim \smallint\limits_{|2^{k}t|<2} |t|^{-(Q-\varkappa)}\phi(2^{k}t)|dt$$ $$=\smallint\limits_{|u|<2} |2^{-k}u|^{-(Q-\varkappa)}\phi(u)|2^{-kQ}du=2^{kQ-k\varkappa-kQ}\smallint\limits_{|u|<2}|u|^{-(Q-\varkappa)}\phi(u)du$$
    $$\lesssim_{\phi}2^{-k\varkappa}.$$
    The analysis above shows the validity of the inequality
    \begin{align}\label{G:B:K}
         \smallint\limits_{G}| G_{\varkappa,k}(t)|dt\leq C_{\phi} 2^{-k\varkappa},
    \end{align} for some $C_{\phi}>0.$
In particular, as $0<1-a<1,$ we have that $|y|\leq |y|^{1-a}.$
Now, we will analyse \eqref{case1:deep:analysis} in three cases. Indeed, for any $k,$ we will analyse the situation when $r=2^{-k}$ is inside of the interval   $[0,|y|/2),$ or, in the interval  $[|y|/2, |y|^{\frac{1-\theta}{1-\alpha}})$  and finally, the case where $r=2^{-k}$ is inside of the set $(|y|^{\frac{1-\theta}{1-\alpha}},\infty).$ See Figure  
\ref{Fig2} below.
\begin{figure}[h]
\includegraphics[width=6cm]{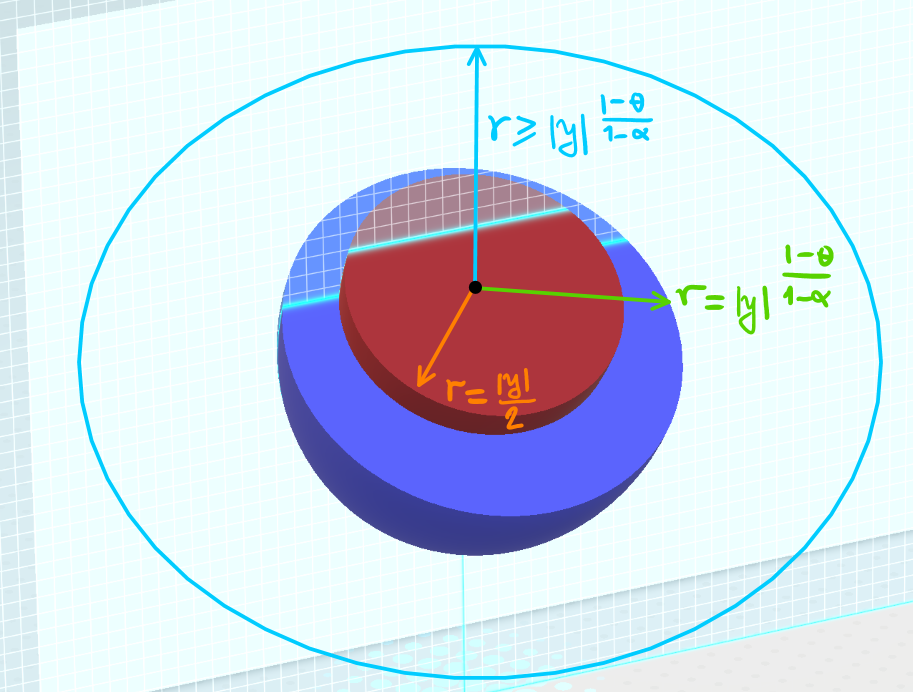}\\
\caption{}
 \label{Fig2}
\centering
\end{figure}
\begin{itemize}
    \item[Case 1:] $2^{-k}<|y|/2.$  In consequence, for the integral in \eqref{case1:deep:analysis}, the inequality $|zt|>2|y|^{1-a}$ implies that $|z|+|t|>2|y|^{1-a}$ and then
$$ |z|>2|y|^{1-a} -|t|>2|y|^{1-a}-2^{-k+1}. $$
The inequality $|y|^{1-a}-|y|\geq 0,$  and the fact that $2^{-k+1}<|y|$ imply that 
$$2|y|^{1-a}-2^{-k+1}>|y|^{1-a}+(|y|^{1-a}-|y|)\geq |y|^{1-a},    $$ and in this case $|z|>|y|^{1-a}.$ We have proved that 
\begin{equation}
    \{z\in G:\forall t\in \textnormal{supp}(G_{\varkappa,k}),|zt|>2|y|^{1-a}\,\}\subset \{z\in G:|z|>|y|^{1-a}\,\}.
\end{equation}So, we can estimate
\begin{align*}
     I_k&\leq  \smallint\limits_{G} |G_{\varkappa,k}(t)| \smallint\limits_{|zt|>2|y|^{1-a}} |K(y^{-1}z)-K(z)|dzdt\\
     &\leq  \smallint\limits_{G} |G_{\varkappa,k}(t)|dt \smallint\limits_{|z|>|y|^{1-a}} |K(y^{-1}z)-K(z)|dz\\
     &\lesssim_{\phi} 2^{-k\varkappa} \smallint\limits_{|z|>|y|^{1-a}} |K(y^{-1}z)-K(z)|dz.
\end{align*}
Let us consider a sequence of points $y_i,$ $0\leq i\leq m,$ $0<1/m<b,$ such that
\begin{equation*}
    y_0=e,\cdots, y_{m}=y, \,d(y_i,y_{i+1})<1/m, \,0\leq i\leq m-1.
\end{equation*}
\begin{itemize}
    \item {\bf{The topological algorithm for the choice of the $y_i$'s.}} For constructing this family of points, we consider the curve
\begin{equation}
    y(t):[0,m]\rightarrow G,\,y(t)=\frac{t}{m}\cdot y,
\end{equation} and the $y_i$'s will belong to its graph. Note that $y(0)=e,$ $y(m)= y,$ and that the derivative $y'(t)$ of the function $y(t)$ is the constant function
$$ y'(t)=\frac{1}{m}\cdot y.    $$

We illustrate the choice of the points $y_i$'s in Figure \ref{Fig1} below.
\begin{figure}[h]
\includegraphics[width=8cm]{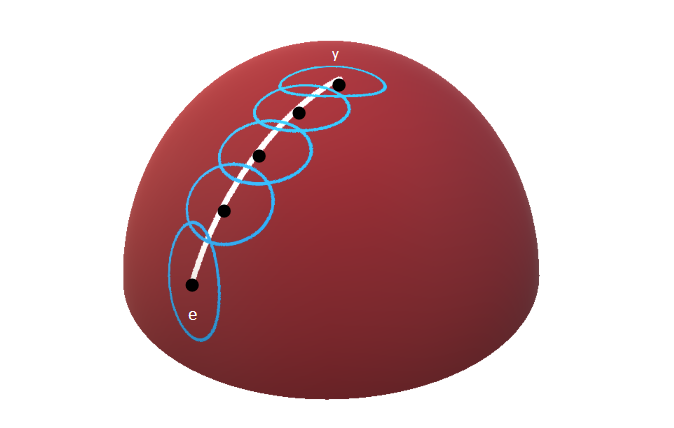}\\
\caption{}
 \label{Fig1}
\centering
\end{figure}
The topological algorithm to choose the points $y_i$ is as follows.  Observe that the length of the curve  $\ell$ is $\leq 1.$ Indeed, $$\ell:=\smallint\limits_{0}^{m}|y'(t)|dt=\smallint\limits_{0}^{m}|1/m\cdot y|dt\leq m(1/m)b\leq 1.$$ Note that we can cover the graph of $y(t)$ with $N_0$ balls $B_i=B(y_i,r_i)$ of radius $r_i=1/m,$ such that $y_0=e,$ $y_{i-i}\in B_{i} $ for $i\geq 2,$   $y_{m}=y,$ and $N_0\sim 2m.$ To guarantee that $d(y_i,y_{i+1})<1/m$ we can take

$$  z_{i+1}\in \partial B_{i}\cap \{y(t):0\leq t\leq m\} $$ and choose $y_{i+1}\in B_i$ such that $d(y_{i+1},z_{i+1})<\frac{1}{2^m}.$ This inductive process ends when one of the balls $B_{i}$ contains the point $y$ in its interior and the distance between $y$ and the center of ball is less than $1/m$.
\end{itemize}
Having fixed the sequence $y_i$  now let us choose a suitable $m.$ Indeed, consider $m\geq 2$ as  the least positive integer such that
$$  \frac{2}{m^{1-\theta}}<|y|^{1-a}-|y|<\frac{2}{(m-1)^{1-\theta}}. $$ 
Then we have that 
$$ |y|^{1-a}-|y|\sim  \frac{2}{m^{1-\theta}}= 2\times \left(\frac{1}{m}\right)^{1-\theta}\sim 2d(y_{i},y_{i+1})^{1-\theta}=2|y_{i}^{-1}y_{i+1}|^{1-\theta}, $$
for all $0\leq i\leq m-1.$ The previous analysis and the changes of variables $x=y_{i-1}^{-1}z$ implies that
$$ I_k\leq  \smallint\limits_{G} |G_{\varkappa,k}(t)|dt \smallint\limits_{|z|>|y|^{1-a}} |K(y^{-1}z)-K(z)|dz $$
$$\lesssim_{\phi}2^{-k\varkappa}\sum_{i=1}^{m}\smallint\limits_{|z|>|y|^{1-a}} |K\left(y_{i}^{-1}z\right)-K\left(y_{i-1}^{-1}z\right)|dz= 2^{-k\varkappa}\sum_{i=1}^m \smallint\limits_{|y_{i-1}\cdot x|>|y|^{1-a}}|K(y_{i}^{-1}y_{i-1}x)-K(x)|dx $$
\begin{align*}
      &\lesssim 2^{-k\varkappa}\sum_{i=1}^m \smallint\limits_{|x|>2|y_{i-1}^{-1}y_{i}|^{1-\theta}}|K(y_{i}^{-1}y_{i-1}x)-K(x)|dx \\ &=2^{-k\varkappa}\sum_{i=1}^m \smallint\limits_{|x|>2|y_{i-1}^{-1}y_{i}|^{1-\theta}}|K((y_{i-1}^{-1}y_{i})^{-1}x)-K(x)|dx \\
      & \lesssim 2^{-k\varkappa}\sum_{i=1}^m[K]_{H_{\infty,\theta,b}}'=2^{-k\varkappa}m[K]_{H_{\infty,\theta,b}}'.
\end{align*}
Indeed, in the previous inequality we have used  the estimate
$$  \smallint\limits_{|y_{i-1}\cdot x|>|y|^{1-a}}|K(y_{i}^{-1}y_{i-1}x)-K(x)|dx\lesssim \smallint\limits_{|x|>2|y_{i-1}^{-1}y_{i}|^{1-\theta} }|K(y_{i}^{-1}y_{i-1}x)-K(x)|dx.   $$ Indeed, estimating $|y_{i-1}|\sim |y|(i-1)/m<|y|,$ we have that the estimate $|y_{i-1}x|\geq |y|^{1-a} $ implies that
$$  |x|>|y|^{1-a}-|y_{i-1}|\succeq |y|^{1-a}-|y|\succeq 2|y_{i-1}^{-1}y_{i}|^{1-\theta}. $$
The choice of $m$ implies that $d(y_i,y_{i+1})\sim \frac{|y|}{m}$ and then 
$$ d(y_i,y_{i+1})^{1-\theta}\sim\left(\frac{|y|}{m}\right)^{1-\theta}\sim |y|^{1-a}-|y|. $$ Then $1/m\sim (|y|^{1-a})^{\frac{1}{1-\theta}}/|y|.$ We then can estimate $m\sim |y|^{1-\frac{1-a}{1-\theta}}.$ So, to finish our analysis in Case 1, note that $|y|^{-1}\lesssim 2^{k}$ which implies that
\begin{align}
    \sum_{k:2^{-k}<|y|/2}I_{k}\lesssim \sum_{k:2^{-k}<|y|/2} 2^{-k\varkappa}m \lesssim |y|^{1-\frac{1-a}{1-\theta}}\sum_{k:2^{-k}<|y|/2}2^{-k\varkappa}\sim |y|^{1-\frac{1-a}{1-\theta}}|y|^{\varkappa}.  
\end{align} Since $\varkappa+1-\frac{1-a}{1-\theta}=0,$ we have that   $ \sum_{k:2^{-k}<|y|/2}I_{k}\lesssim 1.$
    
\item[Case 2:] $|y|/2\leq 2^{-k}<|y|^{\frac{1-\theta}{1-\alpha}}.$ Define $$\delta_k:=5|y|^{2(1-\theta)/\lambda}2^{-kQ(1-\alpha)(1-\theta)/\lambda},$$ where 
$$ \lambda:=Q(1-\theta)+2.  $$
Then we have the upper and the lower bound
$$   5\times {2^{-k}}<\delta_k<5|y|^{1-\theta}.$$
Split $I_k$ as follows,
\begin{align*}
     I_k:=\smallint\limits_{|x|>2|y|^{1-a}}|K\ast G_{\varkappa,k}(y^{-1}x)-K\ast G_{\varkappa,k}(x)|dx=J_{1,k}+J_{2,k},
\end{align*} where
\begin{equation}
    J_{1,k}=\smallint\limits_{\{|x|>2|y|^{1-a} \}\cap \{x:|x|\leq \delta_k\} }|K\ast G_{\varkappa,k}(y^{-1}x)-K\ast G_{\varkappa,k}(x)|dx
\end{equation}and
\begin{equation}
    J_{2,k}=\smallint\limits_{\{|x|>2|y|^{1-a} \}\cap \{x:|x|> \delta_k\} }|K\ast G_{\varkappa,k}(y^{-1}x)-K\ast G_{\varkappa,k}(x)|dx.
\end{equation}Now, let us estimate $J_{2,k}.$ Indeed, the change of variable $z=xt^{-1},$ for $t\in \textnormal{supp}(G_{\varkappa,k})$ implies

\begin{align*}
    J_{2,k}&=\smallint\limits_{ \{|x|>2|y|^{1-a} \}\cap \{x:|x|> \delta_k\}  }|K\ast G_{\varkappa,k}(y^{-1}x)-K\ast G_{\varkappa,k}(x)|dx\\
    &=\smallint\limits_{ \{|x|>2|y|^{1-a} \}\cap \{x:|x|> \delta_k\}  }|\smallint\limits_{G} (K(y^{-1}xt^{-1}) G_{\varkappa,k}(t)-K(xt^{-1}) G_{\varkappa,k}(t))dt|dx\\
    &\leq  \smallint\limits_{G} |G_{\varkappa,k}(t)| \smallint\limits_{ \{|x|>2|y|^{1-a} \}\cap \{x:|x|> \delta_k\}  } |K(y^{-1}xt^{-1}) -K(xt^{-1})|dxdt\\
    &\leq  \smallint\limits_{G} |G_{\varkappa,k}(t)| \smallint\limits_{\{|zt|>2|y|^{1-a} \}\cap \{z:|zt|> \delta_k\}} |K(y^{-1}z)-K(z)|dzdt\\
    &\leq   \smallint\limits_{G} |G_{\varkappa,k}(t)| \smallint\limits_{ \{z:|zt|> \delta_k\}} |K(y^{-1}z)-K(z)|dz dt.
\end{align*}   Note that when  $|zt|>\delta_{k},$ we have $|t|+|z|\geq|zt|>\delta_{k}$ and with $t\in \textnormal{supp}(G_{\varkappa,k}),$ $|t|<2^{-k+1}$ from which ones deduce the inclusion of sets
$$ \{z:|zt|> \delta_k \}\subset\{z: |z|>\delta_k-2^{-k+1}\}, $$
and the estimate
$$  \smallint\limits_{ \{z:|zt|> \delta_k\}} |K(y^{-1}z)-K(z)|dz\leq  \smallint\limits_{ \{z: |z|>\delta_k-2^{-k+1}\} } |K(y^{-1}z)-K(z)|dz.  $$
So, the previous analysis together with \eqref{G:B:K} gives
\begin{equation}
    J_{2,k}\lesssim 2^{-k\varkappa} \smallint\limits_{ \{z:|z|> \delta_k-2^{-k+1}\}} |K(y^{-1}z)-K(z)|dz.
\end{equation}
To continue, let us make use of the argument illustrated in Figure \ref{Fig1}. Using this construction we consider a  sequence of points  $y_i,$ $0\leq i\leq m,$ $0<1/m<b,$ such that
\begin{equation*}
    y_0=e,\cdots, y_{m}=y, \,d(y_i,y_{i+1})\sim |y|/m, \,0\leq i\leq m-1,
\end{equation*}on the curve $y(t)=\frac{t}{m}\cdot y,$ $t\in [0,m],$ and we consider again the topological construction done in Case 1 in order to obtain the required family of points $y_i.$ 
From now, assume that $m$ is the least integer such that
$$ 2d(y_{i},y_{i+1})^{1-\theta}\sim  2(|y|/m)^{1-\theta}<\delta_k-2^{-k+2}. $$
The  changes of variables $x=y_{i-1}^{-1}z$ in any term of the sums below implies that
$$ J_{2,k}\lesssim  2^{-k\varkappa} \smallint\limits_{ \{z:|z|> \delta_k-2^{-k+1}\}} |K(y^{-1}z)-K(z)|dz $$
$$\lesssim_{\phi}2^{-k\varkappa}\sum_{i=1}^{m}\smallint\limits_{ \{z:|z|> \delta_k-2^{-k+1}\}   } |K\left(y_{i}^{-1}z\right)-K\left(y_{i-1}^{-1}z\right)|dz$$
$$= 2^{-k\varkappa}\sum_{i=1}^m \smallint\limits_{ \{x:|y_{i-1}\cdot x|> \delta_k-2^{-k+1}\} }|K(y_{i}^{-1}y_{i-1}x)-K(x)|dx .$$
Note that for $|y_{i-1}\cdot x|> \delta_k-2^{-k+1},$ $|y|+|x|>  \delta_k-2^{-k+1}$ and then, the hypothesis $|y|/2<2^{-k}$
$$ |x|>  \delta_k-2^{-k+1}-|y|> \delta_k-2^{-k+1}-2^{-k+1}=\delta_k-2^{-k+2} \succeq 2d(y_{i},y_{i+1})^{1-\theta}  $$
from which we have proved that $$\smallint\limits_{ \{x:|y_{i-1}\cdot x|> \delta_k-2^{-k+1}\} }|K(y_{i}^{-1}y_{i-1}x)-K(x)|dx\lesssim \smallint\limits_{|x|>2|y_{i-1}^{-1}y_{i}|^{1-\theta}}|K(y_{i}^{-1}y_{i-1}x)-K(x)|dx.$$ In consequence,
\begin{align*}
     J_{2,k}\lesssim & 2^{-k\varkappa}\sum_{i=1}^m \smallint\limits_{|x|>2|y_{i-1}^{-1}y_{i}|^{1-\theta}}|K(y_{i}^{-1}y_{i-1}x)-K(x)|dx \\ &=2^{-k\varkappa}\sum_{i=1}^m \smallint\limits_{|z|>2|y_{i-1}^{-1}y_{i}|^{1-\theta}}|K((y_{i-1}^{-1}y_{i})^{-1}x)-K(x)|dx \\
      & \lesssim 2^{-k\varkappa}\sum_{i=1}^m[K]_{H_{\infty,\theta,b}}'=2^{-k\varkappa}m[K]_{H_{\infty,\theta,b}}'.
\end{align*}It follows that  $m\lesssim |y|^{Q(1-\theta)/\lambda}2^{kQ(1-\theta)/\lambda},$ and in this Case 2,
$$ J_{2,k}\lesssim 2^{-k\varkappa}|y|^{Q(1-\theta)/\lambda}2^{kQ(1-\theta)/\lambda}, $$
 where 
$$ \lambda:=Q(1-\theta)+2.  $$
Now, let us estimate $J_{1,k}.$ In view of the Schwarz inequality we have the estimate:
\begin{align*}
    J_{1,k}&\leq 2\smallint_{|x|\leq \delta_k}|K\ast G_{\varkappa,k}(x)|dx\lesssim \delta_{k}^{\frac{Q}{2}}\|K\ast G_{\varkappa,k}\|_{L^2(G)}=\delta_{k}^{\frac{Q}{2}}\Vert  \widehat{G}_{\varkappa,k}\widehat{K}\Vert_{L^2(\widehat{G})}\\
    &\leq \delta_{k}^{\frac{Q}{2}}\Vert  \widehat{G}_{\varkappa,k} \phi((2^{-k}\cdot \pi)(\mathcal{R})) \widehat{K}\Vert_{L^2(\widehat{G})}
   + \delta_{k}^{\frac{Q}{2}} \Vert  \widehat{G}_{\varkappa,k} (1-\phi)((2^{-k}\cdot \pi)(\mathcal{R})) \widehat{K}\Vert_{L^2(\widehat{G})},
\end{align*} with $\phi$ as in \eqref{the:auxiliar:phi}. Since,
\begin{align*}
    \Vert  \widehat{G}_{\varkappa,k} \phi((2^{-k}\cdot \pi)(\mathcal{R})) \widehat{K}\Vert_{L^2(\widehat{G})}^2 &=\smallint\limits_{\widehat{G}}\| \widehat{G}_{\varkappa,k}(\pi) \phi((2^{-k}\cdot \pi)(\mathcal{R})) \widehat{K}(\pi)\|^2_{\textnormal{HS}}d\pi\\
    &\leq \Vert \widehat{G}_{\varkappa,k}\Vert_{L^\infty(\widehat{G})}^2\smallint\limits_{\widehat{G}}\| \phi((2^{-k}\cdot \pi)(\mathcal{R})) \widehat{K}(\pi)\|^2_{\textnormal{HS}}d\pi.
\end{align*}Using \eqref{G:B:K} we have that $\Vert \widehat{G}_{\varkappa,k}\Vert_{L^\infty(\widehat{G})}^2\leq \Vert {G}_{\varkappa,k} \Vert^2_{L^1(\widehat{G})}\lesssim 2^{-2k\varkappa}$ and then 

$$  \Vert  \widehat{G}_{\varkappa,k} \phi((2^{-k}\cdot \pi)(\mathcal{R})) \widehat{K}\Vert_{L^2(\widehat{G})}^2\lesssim 2^{-2k\varkappa} \| \phi((2^{-k}\cdot \pi)(\mathcal{R})) \widehat{K}(\pi)\|_{L^2(\widehat{G})}.  $$
Using  \eqref{A:alpha}, that is,
\begin{equation}
        \sup_{\pi\in \widehat{G}}\Vert (1+\pi(\mathcal{R}))^{\frac{Q\alpha}{2\nu}}\widehat{K}(\pi) \Vert_{\textnormal{op}}<\infty,
\end{equation}
we have that
$$
   \| \phi((2^{-k}\cdot \pi)(\mathcal{R})) \widehat{K}(\pi)\|_{L^2(\widehat{G})} =\| \phi((2^{-k}\cdot \pi)(\mathcal{R}))(1+\pi(\mathcal{R}))^{-\frac{Q\alpha}{2\nu}} (1+\pi(\mathcal{R}))^{\frac{Q\alpha}{2\nu}}\widehat{K}(\pi)\|_{L^2(\widehat{G})}$$
   $$\leq  \sup_{\pi\in \widehat{G}}\Vert (1+\pi(\mathcal{R}))^{\frac{Q\alpha}{2\nu}}\widehat{K}(\pi) \Vert_{\textnormal{op}}\times  \| \phi((2^{-k}\cdot \pi)(\mathcal{R}))(1+\pi(\mathcal{R}))^{-\frac{Q\alpha}{2\nu}}\|_{L^2(\widehat{G})}$$
   $$\lesssim \| \phi((2^{-k}\cdot \pi)(\mathcal{R}))(1+\pi(\mathcal{R}))^{-\frac{Q\alpha}{2\nu}}\|_{L^2(\widehat{G})}$$
   $$\lesssim \| \phi((2^{-k}\cdot \pi)(\mathcal{R}))\pi(\mathcal{R})^{-\frac{Q\alpha}{2\nu}}\|_{L^2(\widehat{G})}$$
   $$= \| \pi(\mathcal{R})^{-\frac{Q\alpha}{2\nu}} \phi((2^{-k}\cdot \pi)(\mathcal{R}))\|_{L^2(\widehat{G})}.$$
Note that  in the last line we have used the commutativity identity  
$$ \phi((2^{-k}\cdot \pi)(\mathcal{R}))\pi(\mathcal{R})^{-\frac{Q\alpha}{2\nu}}= \pi(\mathcal{R})^{-\frac{Q\alpha}{2\nu}} \phi((2^{-k}\cdot \pi)(\mathcal{R})) $$ in view of the functional calculus of $\mathcal{R},$  and the estimate
\begin{equation}\label{Quotient:L2:act}
     \| \phi((2^{-k}\cdot \pi)(\mathcal{R}))(1+\pi(\mathcal{R}))^{-\frac{Q\alpha}{2\nu}}\|_{L^2(\widehat{G})}
   \lesssim \| \phi((2^{-k}\cdot \pi)(\mathcal{R}))\pi(\mathcal{R})^{-\frac{Q\alpha}{2\nu}}\|_{L^2(\widehat{G})}.
\end{equation}Indeed,
\begin{align*}
    & \| \phi((2^{-k}\cdot \pi)(\mathcal{R}))(1+\pi(\mathcal{R}))^{-\frac{Q\alpha}{2\nu}}\|_{L^2(\widehat{G})}\\
    & = \| \phi((2^{-k}\cdot \pi)(\mathcal{R}))\pi(\mathcal{R})^{-\frac{Q\alpha}{2\nu}}\pi(\mathcal{R})^{\frac{Q\alpha}{2\nu}}(1+\pi(\mathcal{R}))^{-\frac{Q\alpha}{2\nu}}\|_{L^2(\widehat{G})}\\
     &\leq \sup_{\pi\in \widehat{G}} \Vert \pi(\mathcal{R})^{\frac{Q\alpha}{2\nu}}(1+\pi(\mathcal{R}))^{-\frac{Q\alpha}{2\nu}}\Vert_{\textnormal{op}} \| \phi((2^{-k}\cdot \pi)(\mathcal{R}))\pi(\mathcal{R})^{-\frac{Q\alpha}{2\nu}}\|_{L^2(\widehat{G})}\\
     &\lesssim \| \phi((2^{-k}\cdot \pi)(\mathcal{R}))\pi(\mathcal{R})^{-\frac{Q\alpha}{2\nu}}\|_{L^2(\widehat{G})}.
\end{align*} Note that we have used the fact that, in view of the $L^2(G)$-boundedness of the operator $ \mathcal{R}^{\frac{Q\alpha}{2\nu}}(1+\mathcal{R})^{-\frac{Q\alpha}{2\nu}},$ the sup
\begin{equation}\label{sup:quotient}
    \sup_{\pi\in \widehat{G}} \Vert \pi(\mathcal{R})^{\frac{Q\alpha}{2\nu}}(1+\pi(\mathcal{R}))^{-\frac{Q\alpha}{2\nu}}\Vert_{\textnormal{op}}<\infty,
\end{equation}
is finite.  On the other hand, using the Plancherel theorem we get 
\begin{align}
    \| \pi(\mathcal{R})^{-\frac{Q\alpha}{2\nu}} \phi((2^{-k}\cdot \pi)(\mathcal{R})) \|_{L^2(\widehat{G})}=\Vert \mathcal{R}^{-\frac{Q\alpha}{2\nu}}\mathscr{F}_{G}^{-1}[\phi((2^{-k}\cdot \pi)(\mathcal{R}))]  \Vert_{L^2(G)}.
\end{align}With $r=2^{-k},$ and $\Phi_{r}=r^{-Q}\phi(\mathcal{R})\delta(r^{-1}\cdot),$ $\widehat{\Phi}_r(\pi)=\widehat{\Phi}_1(r\cdot \pi).$ In consequence
$$ \phi((2^{-k}\cdot \pi)(\mathcal{R}))=\phi((r\cdot \pi)(\mathcal{R}))= \widehat{\phi(\mathcal{R}\delta)}(r\cdot \pi)= \widehat{\Phi}_1(r\cdot \pi) $$
and 
\begin{align*}
    \mathcal{R}^{-\frac{Q\alpha}{2\nu}}\mathscr{F}_{G}^{-1}[\phi((2^{-k}\cdot \pi)(\mathcal{R}))]=\mathcal{R}^{-\frac{Q\alpha}{2\nu}}\mathscr{F}_{G}^{-1}[\phi((r\cdot \pi)(\mathcal{R}))]=\mathcal{R}^{-\frac{Q\alpha}{2\nu}}\mathscr{F}_{G}^{-1}[\widehat{\Phi}_r(\pi)]
\end{align*}
$$ =\mathcal{R}^{-\frac{Q\alpha}{2\nu}}\Phi_r.  $$
As $0<Q\alpha/2<Q,$ in view of Corollary 4.3.11 of \cite{FischerRuzhanskyBook}, the right-convolution kernel of $\mathcal{R}^{-\frac{Q\alpha}{2\nu}}$ is homogeneous  of order $\frac{Q\alpha}{2}-Q,$ and in consequence of \cite[Lemma 3.2.7]{FischerRuzhanskyBook} $\mathcal{R}^{-\frac{Q\alpha}{2\nu}}$ has homogeneous degree equal to $-Q\alpha/2.$ So, we have that 
\begin{align*}
    \Vert \mathcal{R}^{-\frac{Q\alpha}{2\nu}}\mathscr{F}_{G}^{-1}[\phi((2^{-k}\cdot \pi)(\mathcal{R}))]  \Vert_{L^2(G)}&=\Vert \mathcal{R}^{-\frac{Q\alpha}{2\nu}}\Phi_r  \Vert_{L^2(G)}=r^{-Q}\Vert \mathcal{R}^{-\frac{Q\alpha}{2\nu}}[\phi(\mathcal{R})\delta(r^{-1}\cdot)]  \Vert_{L^2(G)}\\
    &=r^{-Q} r^{\frac{Q\alpha}{2}}\Vert \mathcal{R}^{-\frac{Q\alpha}{2\nu}}[\phi(\mathcal{R})\delta](r^{-1}\cdot) \Vert_{L^2(G)}\\
    &=r^{-Q} r^{\frac{Q\alpha}{2}}r^{\frac{Q}{2}}\Vert \mathcal{R}^{-\frac{Q\alpha}{2\nu}}[\phi(\mathcal{R})\delta](\cdot) \Vert_{L^2(G)}\\
    &=2^{-k(\frac{Q\alpha}{2}-\frac{Q}{2})}\Vert \mathcal{R}^{-\frac{Q\alpha}{2\nu}}[\phi(\mathcal{R})\delta](\cdot) \Vert_{L^2(G)}.
\end{align*}In view of the Hulanicki theorem in \cite{FischerRuzhanskyBook}, $\phi(\mathcal{R})\delta\in \mathscr{S}(G)$ and then $$ \Vert \mathcal{R}^{-\frac{Q\alpha}{2\nu}}[\phi(\mathcal{R})\delta](\cdot) \Vert_{L^2(G)}<\infty,$$ in view of Corollary 4.3.11 in \cite{FischerRuzhanskyBook}. All the analysis above implies that
\begin{align*}
    J_{1,k}
    &\leq \delta_{k}^{\frac{Q}{2}}\Vert  \widehat{G}_{\varkappa,k} \phi((2^{-k}\cdot \pi)(\mathcal{R})) \widehat{K}\Vert_{L^2(\widehat{G})}
   + \delta_{k}^{\frac{Q}{2}} \Vert  \widehat{G}_{\varkappa,k} (1-\phi)((2^{-k}\cdot \pi)(\mathcal{R})) \widehat{K}\Vert_{L^2(\widehat{G})}\\
   &\lesssim \delta_{k}^{\frac{Q}{2}}2^{-k\varkappa}2^{-k(\frac{Q\alpha}{2}-\frac{Q}{2})}+ \delta_{k}^{\frac{Q}{2}} \Vert  \widehat{G}_{\varkappa,k} (1-\phi)((2^{-k}\cdot \pi)(\mathcal{R})) \widehat{K}\Vert_{L^2(\widehat{G})}\\
   &\lesssim \delta_{k}^{\frac{Q}{2}}2^{-k(\varkappa+\frac{Q(\alpha-1)}{2})}+ \delta_{k}^{\frac{Q}{2}} \Vert  \widehat{G}_{\varkappa,k} (1-\phi)((2^{-k}\cdot \pi)(\mathcal{R})) \widehat{K}\Vert_{L^2(\widehat{G})}\\
   &=\delta_{k}^{\frac{Q}{2}}2^{-\frac{kQ(a-1)}{2}}+ \delta_{k}^{\frac{Q}{2}} \Vert  \widehat{G}_{\varkappa,k} (1-\phi)((2^{-k}\cdot \pi)(\mathcal{R})) \widehat{K}\Vert_{L^2(\widehat{G})}.
\end{align*}Now, we will prove the estimate
\begin{align}\label{Aux:again:jik}
    \Vert  \widehat{G}_{\varkappa,k} (1-\phi)((2^{-k}\cdot \pi)(\mathcal{R})) \widehat{K}\Vert_{L^2(\widehat{G})}\lesssim 2^{-\frac{kQ(a-1)}{2}},
\end{align} in order to have the following upper bound for $J_{1,k},$
\begin{equation}\label{to:proof:jik}
   J_{1,k} \lesssim\delta_{k}^{\frac{Q}{2}}2^{-k(\varkappa+\frac{Q(\alpha-1)}{2})}=\delta_{k}^{\frac{Q}{2}}2^{-\frac{kQ(1-a)}{2})}.
\end{equation}For the proof of \eqref{Aux:again:jik} note that
\begin{align*}
     &\Vert  \widehat{G}_{\varkappa,k} (1-\phi)((2^{-k}\cdot \pi)(\mathcal{R})) \widehat{K}\Vert_{L^2(\widehat{G})}\\
     &= \Vert  \widehat{G}_{\varkappa,k} (1-\phi)((2^{-k}\cdot \pi)(\mathcal{R}))(1+\pi(\mathcal{R}))^{-\frac{Q\alpha}{2\nu}}(1+\pi(\mathcal{R}))^{\frac{Q\alpha}{2\nu}} \widehat{K}\Vert_{L^2(\widehat{G})}\\
     &\leq  \sup_{\pi\in \widehat{G}}\Vert (1+\pi(\mathcal{R}))^{\frac{Q\alpha}{2\nu}} \widehat{K}(\pi)\Vert_{\textnormal{op}} 
     \Vert  \widehat{G}_{\varkappa,k} (1-\phi)((2^{-k}\cdot \pi)(\mathcal{R}))(1+\pi(\mathcal{R}))^{-\frac{Q\alpha}{2\nu}}\Vert_{L^2(\widehat{G})}\\
     &\lesssim 
     \Vert  \widehat{G}_{\varkappa,k} (1-\phi)((2^{-k}\cdot \pi)(\mathcal{R}))(1+\pi(\mathcal{R}))^{-\frac{Q\alpha}{2\nu}}\Vert_{L^2(\widehat{G})}.
\end{align*}Using again the estimate in \eqref{sup:quotient} we have that
\begin{align*}
    & \Vert  \widehat{G}_{\varkappa,k} (1-\phi)((2^{-k}\cdot \pi)(\mathcal{R}))(1+\pi(\mathcal{R}))^{-\frac{Q\alpha}{2\nu}}\Vert_{L^2(\widehat{G})}\\
    &=\Vert  \widehat{G}_{\varkappa,k} (1-\phi)((2^{-k}\cdot \pi)(\mathcal{R})) \pi(\mathcal{R})^{-\frac{Q\alpha}{2\nu}} \pi(\mathcal{R})^{\frac{Q\alpha}{2\nu}}  (1+\pi(\mathcal{R}))^{-\frac{Q\alpha}{2\nu}}\Vert_{L^2(\widehat{G})}\\
     &\lesssim \Vert  \widehat{G}_{\varkappa,k} (1-\phi)((2^{-k}\cdot \pi)(\mathcal{R}))\pi(\mathcal{R})^{-\frac{Q\alpha}{2\nu}}\Vert_{L^2(\widehat{G})}.
\end{align*}Now, let us use the functional calculus of $\mathcal{R}.$ For any continuous function $\kappa(t)$ on $\mathbb{R}^+$ one has that
\begin{equation}
 \forall r>0,   \kappa(r^{\nu}\mathcal{R})\delta=r^{-Q}[ \kappa(\mathcal{R})\delta](r^{-1}\cdot).
\end{equation}Taking in both sides the group Fourier transform one has
\begin{align*}
     \kappa(r^{\nu}\pi(\mathcal{R}))=  \kappa((r\cdot \pi)(\mathcal{R})).
\end{align*}The previous identity with $\kappa(t)=t^{-\frac{Q}{2\nu}}$ gives
\begin{equation}
\forall r>0,     (r^{\nu}\pi(\mathcal{R}))^{-\frac{Q\alpha}{2\nu}}=((r\cdot \pi)(\mathcal{R}) )^{-\frac{Q\alpha}{2\nu}}.
\end{equation}Using the previous property, and the changes of variables $\pi'=2^{-k}\cdot \pi,$ we have the effect in the Borel measure $d\pi'=2^{-kQ}d\pi$ on the unitary dual $\widehat{G}$ and we can estimate 
\begin{align*}
  & \Vert  \widehat{G}_{\varkappa,k} (1-\phi)((2^{-k}\cdot \pi)(\mathcal{R}))\pi(\mathcal{R})^{-\frac{Q\alpha}{2\nu}}\Vert_{L^2(\widehat{G})}^2\\
  &=\smallint\limits_{\widehat{G}}\Vert \widehat{G}_{\varkappa,k}(\pi) (1-\phi)((2^{-k}\cdot \pi)(\mathcal{R}))\pi(\mathcal{R})^{-\frac{Q\alpha}{2\nu}}  \Vert_{\textnormal{HS}}^2d\pi\\
   &=\smallint\limits_{\widehat{G}}\Vert \widehat{G}_{\varkappa,k}(2^{k}\cdot \pi') (1-\phi)(\pi'(\mathcal{R}))((2^k\cdot\pi')(\mathcal{R}))^{-\frac{Q\alpha}{2\nu}}  \Vert_{\textnormal{HS}}^22^{kQ}d\pi'\\
   &=\smallint\limits_{\widehat{G}}\Vert \widehat{G}_{\varkappa,k}(2^{k}\cdot \pi') (1-\phi)(\pi'(\mathcal{R}))(2^{k\nu}\pi(\mathcal{R}))^{-\frac{Q\alpha}{2\nu}}  \Vert_{\textnormal{HS}}^22^{kQ}d\pi'\\
   &=\smallint\limits_{\widehat{G}}\Vert \widehat{G}_{\varkappa,k}(2^{k}\cdot \pi') (1-\phi)(\pi'(\mathcal{R}))(\pi'(\mathcal{R}))^{-\frac{Q\alpha}{2\nu}}  \Vert_{\textnormal{HS}}^22^{k(Q-Q\alpha)}d\pi'\\
  & \lesssim\smallint\limits_{\widehat{G}}\Vert \widehat{G}_{\varkappa,k}(2^{k}\cdot \pi') (1-\phi)(\pi'(\mathcal{R}))(1+\pi'(\mathcal{R}))^{-\frac{Q\alpha}{2\nu}}  \Vert_{\textnormal{HS}}^22^{k(Q-Q\alpha)}d\pi'.
\end{align*}
Then, we have estimated
\begin{align*}
    & \Vert  \widehat{G}_{\varkappa,k} (1-\phi)((2^{-k}\cdot \pi)(\mathcal{R}))\pi(\mathcal{R})^{-\frac{Q\alpha}{2\nu}}\Vert_{L^2(\widehat{G})}^2\\ &\lesssim\smallint\limits_{\widehat{G}}\Vert \widehat{G}_{\varkappa,k}(2^{k}\cdot \pi') (1-\phi)(\pi'(\mathcal{R}))(1+\pi'(\mathcal{R}))^{-\frac{Q\alpha}{2\nu}}  \Vert_{\textnormal{HS}}^22^{k(Q-Q\alpha)}d\pi'.
\end{align*}
Now, let us use the identity
\begin{align*}
    (1-\phi)=(1-\phi)^2+\phi(1-\phi).
\end{align*}We have that
\begin{align*}
   & \Vert \widehat{G}_{\varkappa,k}(2^{k}\cdot \pi') (1-\phi)(\pi'(\mathcal{R}))(1+\pi'(\mathcal{R}))^{-\frac{Q\alpha}{2\nu}}  \Vert_{L^2(\widehat{G})}\\
    &\leq \Vert \widehat{G}_{\varkappa,k}(2^{k}\cdot \pi') (1-\phi)^2(\pi'(\mathcal{R}))(1+\pi'(\mathcal{R}))^{-\frac{Q\alpha}{2\nu}}  \Vert_{L^2(\widehat{G})}\\
    &\hspace{3cm}+\Vert \widehat{G}_{\varkappa,k}(2^{k}\cdot \pi') \phi(1-\phi)(\pi'(\mathcal{R}))(1+\pi'(\mathcal{R}))^{-\frac{Q\alpha}{2\nu}}  \Vert_{L^2(\widehat{G})}=R_{1}+R_{2}.
\end{align*}
Let us estimate $R_2,$ that is the last term of the previous inequality.
\begin{align*}
    R_2 &=\Vert \widehat{G}_{\varkappa,k}(2^{k}\cdot \pi') \phi(1-\phi)(\pi'(\mathcal{R}))(1+\pi'(\mathcal{R}))^{-\frac{Q\alpha}{2\nu}}  \Vert_{L^2(\widehat{G})}\\
    &\lesssim\Vert \widehat{G}_{\varkappa,k}\Vert_{L^\infty(\widehat{G})}\Vert (1+\pi'(\mathcal{R}))^{-\frac{Q\alpha}{2\nu}}   \phi(1-\phi)(\pi'(\mathcal{R})) \Vert_{L^2(\widehat{G})}\\
    &\lesssim \Vert {G}_{\varkappa,k}\Vert_{L^{1}(G)} \Vert (1+\mathcal{R})^{-\frac{Q\alpha}{2\nu}}[ [\phi(1-\phi)](\mathcal{R})\delta \Vert_{L^2(G)}.
\end{align*}In view of the Hulanicki theorem in \cite{FischerRuzhanskyBook}, we have that $[\phi(1-\phi)](\mathcal{R})\delta \in \mathscr{S}(G),$ and 
\begin{align*}
     \Vert (1+\mathcal{R})^{-\frac{Q\alpha}{2\nu}}[ [\phi(1-\phi)](\mathcal{R})\delta \Vert_{L^2(G)}=\Vert \phi(1-\phi)](\mathcal{R})\delta \Vert_{L^2_{ -\frac{Q\alpha}{2}}(G)}<\infty.
\end{align*}So, we have proved that
$$ R_2\lesssim \Vert {G}_{\varkappa,k}\Vert_{L^{1}(G)}\lesssim2^{-k\varkappa}.  $$

Now, let $N=n_{0}\nu>Q/2,$ where $n_0\in \mathbb{N}.$ Let us consider and let $\mathcal{B}_N$ be the Bessel potential defined by $\widehat{\mathcal{B}}_N(\pi)=(1+\pi(\mathcal{R}))^{\frac{N}{\nu}}.$ We can estimate
\begin{align*}
   & R_1= \smallint\limits_{\widehat{G}}\Vert \widehat{G}_{\varkappa,k}(2^{k}\cdot \pi') (1-\phi)^2(\pi'(\mathcal{R}))  (1+\pi'(\mathcal{R}))^{-\frac{Q\alpha}{2\nu}} \|_{\textnormal{HS}}^2d\pi'\\
   &\leq \smallint\limits_{\widehat{G}}\Vert \widehat{G}_{\varkappa,k}(2^{k}\cdot \pi') (1-\phi)^2(\pi'(\mathcal{R}))   \|_{\textnormal{HS}}^2d\pi'\\
    &=\smallint\limits_{\widehat{G}}\Vert \widehat{G}_{\varkappa,k}(2^{k}\cdot \pi') (1-\phi)(\pi'(\mathcal{R}))(1+\pi'(\mathcal{R}))^{\frac{N}{\nu}}(1-\phi)(\pi'(\mathcal{R}))(1+\pi'(\mathcal{R}))^{-\frac{N}{\nu}}\|_{\textnormal{HS}}^2d\pi'.
\end{align*} Note that the pseudo-differential operator $(1-\phi)(\mathcal{R})(1+\mathcal{R})^{-\frac{N}{\nu}}$ is smoothing and then its right-convolution kernel $k_{N}$   belongs to the Schwartz space $\mathscr{S}(G).$ Note also that
$$ \|  (1-\phi)(\pi'(\mathcal{R}))\|_{L^\infty(\widehat{G})}=\sup_{\pi'\in \widehat{G}} \Vert  (1-\phi)(\pi'(\mathcal{R}))\Vert_{\textnormal{op}}\leq \Vert 1-\phi\Vert_{L^{\infty}(\mathbb{R}^+)}\lesssim 1, $$
in view of the Functional calculus of the operator $\pi'(\mathcal{R}),$ $\pi'\in \widehat{G},$ and the properties of $\phi$ in \eqref{the:auxiliar:phi}. So, using the Plancherel theorem  we estimate
\begin{align*}
     & R_1\\
    &=\smallint\limits_{\widehat{G}}\Vert \widehat{G}_{\varkappa,k}(2^{k}\cdot \pi') (1-\phi)(\pi'(\mathcal{R}))(1+\pi'(\mathcal{R}))^{\frac{N}{\nu}}(1-\phi)(\pi'(\mathcal{R}))(1+\pi'(\mathcal{R}))^{-\frac{N}{\nu}}\|_{\textnormal{HS}}^2d\pi'\\
  &\leq \Vert \widehat{G}_{\varkappa,k}\Vert^2_{L^{\infty}(\widehat{G})}\|  (1-\phi)(\pi'(\mathcal{R}))\|^2_{L^\infty(\widehat{G})}\smallint\limits_{\widehat{G}}\Vert (1+\pi'(\mathcal{R}))^{\frac{N}{\nu}} (1-\phi)(\pi'(\mathcal{R}))(1+\pi'(\mathcal{R}))^{-\frac{N}{\nu}}\|_{\textnormal{HS}}^2d\pi'\\
  &\leq \Vert \widehat{G}_{\varkappa,k}\Vert^2_{L^{\infty}(\widehat{G})}\|  (1-\phi)(\pi'(\mathcal{R}))\|^2_{L^\infty(\widehat{G})}\smallint\limits_{\widehat{G}}\Vert (1+\pi'(\mathcal{R}))^{\frac{N}{\nu}}\widehat{k}_N(\pi')\|_{\textnormal{HS}}^2d\pi'\\
  &\lesssim 2^{-2k\varkappa}\smallint\limits_{\widehat{G}}\Vert(1+\pi'(\mathcal{R}))^{\frac{N}{\nu}}\widehat{k}_N(\pi')\|_{\textnormal{HS}}^2d\pi'= 2^{-2k\varkappa}\smallint\limits_{\widehat{G}}\Vert (1+\mathcal{R})^{\frac{N}{\nu}}k_N\|_{L^2(G)}^2.
\end{align*}So, we have proved that
$$ R_1\lesssim \Vert {G}_{\varkappa,k}\Vert_{L^{1}(G)}\lesssim2^{-k\varkappa}.  $$
The analysis above allows us to conclude that
\begin{align*}
    \Vert  \widehat{G}_{\varkappa,k} (1-\phi)((2^{-k}\cdot \pi)(\mathcal{R})) \widehat{K}\Vert_{L^2(\widehat{G})}\lesssim 2^{-k(\frac{Q(1-\alpha)}{2} +\varkappa)}= 2^{\frac{kQ(1-a)}{2}},
\end{align*}as well as the estimate \eqref{to:proof:jik}. It follows then that
$$ J_{1,k}\lesssim 2^{-k\varkappa}|y|^{Q(1-\theta)/\lambda}2^{kQ(1-\theta)/\lambda}.  $$
So, to finish our proof  in Case 2, note that $|y|^{-1}\lesssim 2^{k}$ which implies that
\begin{align}
    \sum_{k:|y|/2\leq 2^{-k}<|y|^{\frac{1-\theta}{1-\alpha}}}I_{k}\lesssim \sum_{k:|y|/2\leq 2^{-k}<|y|^{\frac{1-\theta}{1-\alpha}}} 2^{-k\varkappa}|y|^{Q(1-\theta)/\lambda}2^{kQ(1-\theta)/\lambda}\lesssim 1.  
\end{align} 
\item[Case 3:] $|y|^{\frac{1-\theta}{1-\alpha}}\leq 2^{-k}.$  Define $$\delta_k:=4\cdot2^{-k(1-\alpha)}.$$ Note that
$$  \delta_k/2\geq 2|y|^{1-\theta}. $$
Split $I_k$ as follows,
\begin{align*}
     I_k:=\smallint\limits_{|x|>2|y|^{1-a}}|K\ast G_{\varkappa,k}(y^{-1}x)-K\ast G_{\varkappa,k}(x)|dx=J_{1,k}+J_{2,k},
\end{align*} where
\begin{equation}
    J_{1,k}=\smallint\limits_{\{|x|>2|y|^{1-a} \}\cap \{x:|x|\leq \delta_k\} }|K\ast G_{\varkappa,k}(y^{-1}x)-K\ast G_{\varkappa,k}(x)|dx
\end{equation}and
\begin{equation}
    J_{2,k}=\smallint\limits_{\{|x|>2|y|^{1-a} \}\cap \{x:|x|> \delta_k\} }|K\ast G_{\varkappa,k}(y^{-1}x)-K\ast G_{\varkappa,k}(x)|dx.
\end{equation}Now, let us estimate $J_{2,k}.$ Indeed, the change of variable $z=xt^{-1},$ for $t\in \textnormal{supp}(G_{\varkappa,k})$ implies

\begin{align*}
    J_{2,k}&=\smallint\limits_{ \{|x|>2|y|^{1-a} \}\cap \{x:|x|> \delta_k\}  }|K\ast G_{\varkappa,k}(y^{-1}x)-K\ast G_{\varkappa,k}(x)|dx\\
    &=\smallint\limits_{ \{|x|>2|y|^{1-a} \}\cap \{x:|x|> \delta_k\}  }|\smallint\limits_{G} (K(y^{-1}xt^{-1}) G_{\varkappa,k}(t)-K(xt^{-1}) G_{\varkappa,k}(t))dt|dx\\
    &\leq  \smallint\limits_{G} |G_{\varkappa,k}(t)| \smallint\limits_{ \{|x|>2|y|^{1-a} \}\cap \{x:|x|> \delta_k\}  } |K(y^{-1}xt^{-1}) -K(xt^{-1})|dxdt\\
    &\leq  \smallint\limits_{G} |G_{\varkappa,k}(t)| \smallint\limits_{\{|zt|>2|y|^{1-a} \}\cap \{z:|zt|> \delta_k\}} |K(y^{-1}z)-K(z)|dzdt\\
    &\leq   \smallint\limits_{G} |G_{\varkappa,k}(t)| \smallint\limits_{ \{z:|zt|> \delta_k\}} |K(y^{-1}z)-K(z)|dz dt.
\end{align*}   Note that when  $|zt|>\delta_{k},$ we have $|t|+|z|\geq|zt|>\delta_{k}$ and with $t\in \textnormal{supp}(G_{\varkappa,k}),$ $|t|<2^{-k+1}$ from which ones deduce the inclusion of sets
$$ \{z:|zt|> \delta_k \}\subset\{z: |z|>\delta_k-2^{-k+1}\}, $$
and the estimate
$$  \smallint\limits_{ \{z:|zt|> \delta_k\}} |K(y^{-1}z)-K(z)|dz\leq  \smallint\limits_{ \{z: |z|>\delta_k-2^{-k+1}\} } |K(y^{-1}z)-K(z)|dz  $$
$$ \leq  \smallint\limits_{ \{z: |z|>\delta_k/2\} } |K(y^{-1}z)-K(z)|dz  $$
$$ \leq  \smallint\limits_{ \{z: |z|>2|y|^{1-\theta}\} } |K(y^{-1}z)-K(z)|dz , $$where we have used that $\delta_k/2\geq 2|y|^{1-\theta}.$
So, the previous analysis together with \eqref{G:B:K} gives
\begin{equation}
    J_{2,k}\lesssim 2^{-k\varkappa} \smallint\limits_{ \{z: |z|>2|y|^{1-\theta}\} } |K(y^{-1}z)-K(z)|dz\lesssim 2^{-k\varkappa}.
\end{equation}The same analysis done in Case 2, allows us to deduce the estimate
$$ J_{1,k}\leq C\delta_{k}^{\frac{Q}{2}}2^{-kQ(a-1)/2} \leq C2^{-k\varkappa}. $$
So, to finish Case 3, note that
\begin{align}
    \sum_{k:|y|^{\frac{1-\theta}{1-\alpha}}<2^{-k} }I_{k}\lesssim \sum_{k:|y|^{\frac{1-\theta}{1-\alpha}}<2^{-k}} 2^{-k\varkappa}\lesssim 1.  
\end{align} 
\end{itemize}
The proof of Theorem \ref{Sjolin:Th:graded:groups} is complete.
\end{proof}

\bibliographystyle{amsplain}

\end{document}